\documentclass{article}

\usepackage{microtype}
\usepackage{graphicx}
\usepackage{subfigure}
\usepackage{booktabs} 

\usepackage{setspace}
\usepackage{xspace}

\usepackage[utf8]{inputenc} 
\usepackage[T1]{fontenc}    
\usepackage{hyperref}       
\usepackage{url}            
\usepackage{booktabs}       
\usepackage{amsfonts}       
\usepackage{nicefrac}       
\usepackage{microtype}      
\usepackage{xcolor}         
\usepackage[normalem]{ulem}

\newcommand{\calA}{\mathcal{A}}

\newcommand{\infeas}{\textsc{Infeasible}\xspace}
\newcommand{\feas}{\textsc{Feasible}\xspace}
\newcommand{\feasbar}{\overline{\textsc{Feas}}}
\newcommand{\err}{err}
\newcommand{\cI}{\mathcal{I}}
\newcommand{\F}{\mathcal{F}}
\newcommand{\remove}[1]{}

\newcommand{\pk}[1]{\textcolor{teal}{#1}}
\usepackage{hyperref}

\usepackage[accepted]{icml2024}

\usepackage{extra_sty_for_neurips_version}

\usepackage[capitalize,noabbrev]{cleveref}

\usepackage[textsize=tiny]{todonotes}

\icmltitlerunning{A Universal Transfer Theorem for Convex Optimization}

\begin{document}

\twocolumn[
\icmltitle{A Universal Transfer Theorem for Convex Optimization Algorithms \texorpdfstring{\\}{}Using Inexact First-order Oracles}



\icmlsetsymbol{equal}{*}

\begin{icmlauthorlist}
\icmlauthor{Phillip Kerger}{equal,jhu}
\icmlauthor{Marco Molinaro}{equal,msr,puc}
\icmlauthor{Hongyi Jiang}{cornell}
\icmlauthor{Amitabh Basu}{jhu}
\end{icmlauthorlist}

\icmlaffiliation{jhu}{Department of Applied Mathematics and Statistics, Johns Hopkins University, Baltimore, USA}
\icmlaffiliation{cornell}{Department of Civil and Environmental Engineering, Cornell University, Ithaca, USA}
\icmlaffiliation{msr}{Microsoft Research, Redmond, USA}
\icmlaffiliation{puc}{Department of Computer Science, PUC-Rio, Rio de Janeiro, Brazil}

\icmlcorrespondingauthor{Amitabh Basu,}{abasu9@jhu.edu}
\icmlcorrespondingauthor{Hongyi Jiang}{hj348@cornell.edu}
\icmlcorrespondingauthor{Phillip Kerger}{pkerger@jhu.edu}
\icmlcorrespondingauthor{Marco Molinaro}{mmolinaro@microsoft.com}

\icmlkeywords{Convex Optimization, First-order methods, inexact information, approximate information, Smooth optimization}

\vskip 0.3in
]



\printAffiliationsAndNotice{\icmlEqualContribution} 

\begin{abstract}
Given {\em any} algorithm for convex optimization that uses exact first-order information (i.e., function values and subgradients), we show how to use such an algorithm to solve the problem with access to \emph{inexact} first-order information. This is done in a ``black-box'' manner without knowledge of the internal workings of the algorithm. This complements previous work that
considers the performance of specific algorithms like (accelerated) gradient descent with inexact information. In particular, our results apply to a wider range of algorithms beyond variants of gradient descent, e.g., projection-free methods, cutting-plane methods, or any other first-order methods formulated in the future. Further, they also apply to algorithms that handle structured nonconvexities like mixed-integer decision variables.
\end{abstract}

\section{Introduction}\label{sec:introduction}

Optimization is a core tool for almost any learning or estimation problem. Such problems are very often approached by setting up an optimization problem whose decision variables model the entity to be estimated, and whose objective and constraints are defined by the observed data combined with structural insights into the inference problem. Algorithms for any sufficiently general class of relevant optimization problems in such settings need to collect information about the particular instance by making (adaptive) queries about the objective  before they can report a good solution. In this paper, we focus on the following important class of optimization problems {\em over a fixed ground set $X \subseteq \R^d$}
\begin{equation}\label{eq:main-prob}
    \min\{f(x) : x \in X\},
\end{equation}
where $f:\R^d \to \R$ is a (possibly nonsmooth) convex function. When the underlying ground set $X$ is all of $\R^d$ or some fixed convex subset,~\eqref{eq:main-prob} is the classical convex optimization problem. In this paper, we allow $X$ to be more general and to be used to model some known nonconvexity, e.g. integrality constraints by setting $X= C \cap (\Z^{d_1} \times \R^{d_2})$ with $d_1 + d_2 = d$, where $C\subseteq \R^d$ is a fixed convex set. From an algorithmic perspective, the setup is that the algorithm has complete knowledge of what $X$ is, but does not {\em a priori} know $f$ and must collect information via queries.
A standard model for accessing the function $f$ is through so-called {\em first-order oracles}. At any point during its execution, the algorithm can request the function value and the (sub)gradient of $f$ at any point $\bar{x} \in \R^d$. 

Given access to such oracles, a \emph{first-order algorithm} makes adaptive queries to this oracle 
and, after it judges that it has collected enough information about $f$, it reports a solution with certain guarantees. A long line of research has gone into understanding exactly how many queries are needed to solve different classes of problems (with different sets of assumptions on $f$ and $X$), with tight upper and lower bounds on the  query complexity (a.k.a. {\em oracle} or {\em information complexity}) known in the literature; see~\cite{Nesterov-Book04,bubeck2014convex,nemirovski1994efficient,basu2021complexity,basu2023information} for expositions of these results.

A natural question that arises in this context is what happens if the response of the oracle is not exact, but approximate (with possibly desired accuracy). For example, 
the response of the oracle might be itself a solution to another computational problem which is solved only approximately, which happens when using function smoothing \cite{Nesterov2005_Smooth_minimization_of_nonsmooth}, and 
in minimax problems~\cite{wang2018acceleration}. Stochastic first-order oracles, modeling applications where only some estimate of the gradient is used, may also be viewed as inexact oracles whose accuracy is a random variable at each iteration. Thus, researchers have also investigated what one can say about algorithms that have access to inexact oracle responses (with possibly known guarantees on the inexactness). Early work on this topic appears in Shor~\cite{shor1985minimization} and Polyak~\cite{polyak1987introduction}, and more recent progress can be found in~\cite{Devolder_Nesterov2014_First_order_inexact_oracles,schmidt2011convergence,lan2009convex,hintermuller2001proximal,kiwiel2006proximal,d2008smooth} and references therein. To the best of our knowledge, all previous work on inexact first-order oracles has focused either on how specific algorithms like (accelerated) gradient methods perform with inexact (sub)gradients  with no essential change to the algorithm, or on how to adapt a particular class of algorithms to perform well with inexact information. 

In this paper, we provide a different approach to the problem of inexact information. We provide a way to take {\em any} first-order algorithm that solves~\eqref{eq:main-prob} with exact first-order information and, with absolutely no knowledge of the its inner workings, show how to make the same algorithm work with inexact oracle information. 
Thus, in contrast to earlier work, our result is not the analysis of specific algorithms under inexact information or the adaptation of specific algorithms to use inexact information.
It is in this sense that we believe our results to be {\em universal} because they apply to a much wider class of algorithms than previous work, including gradient descent, cutting plane methods, bundle methods, projection-free methods etc., and also to any first-order method that is invented in the future for optimization problems of the form~\eqref{eq:main-prob}. 



\section{Formal statement of results and discussion}\label{sec:results}

We begin with definitions of standard concepts that we need to state our results formally. We use $\|\cdot\|$ to denote the standard Euclidean norm and $B(c,r)$ to denote the Euclidean ball of radius $r$ centered at $c\in \R^d$. When the center is the origin, we denote the ball by $B(r)$. A function $h: \R^d \rightarrow \R$ is said to be $M$-Lipschitz if $|h(x) - h(y)| \le M \|x - y\|$ for all $x,y$.  Let $\F^0(M,R)$ denote the standard family of instances of the optimization problem \eqref{eq:main-prob} consisting of all $M$-Lipschitz (possibly non-differentiable) convex functions $f$ 
 such that the minimizer $x^\star \in X$ is contained in the ball $B(R)$.\footnote{This is a standard assumption in the analysis of optimization algorithm -- if no such bound is assumed, then it can be shown that no algorithm can report a good solution within a guaranteed number of steps for every instance~\cite{Nesterov-Book04}. Alternatively, one may give the convergence rates in terms of the distance of the initial iterate of the algorithm and the optimal solution (one can think of $R$ as an upper bound on this distance). Our results can also be formulated in this language with no conceptual or technical changes.}
%
%
We now formalize the inexact first-order oracles that we will work with.

\begin{definition}\label{def:approxFOO}
An \emph{$\eta$-approximate first-order oracle} for a convex function $f: \R^d \rightarrow \R$ takes as input a query point $\bar{x} \in \R^d$ and returns a first-order pair $(\tilde{f}, \tilde{g})$ satisfying $|\tilde{f} - f(\bar{x})|\leq \eta$ and $\norm{\tilde{g} - g} \leq \frac{\eta}{2R}$ for some subgradient $g \in \partial f(\bar{x})$. 
\end{definition}

We now state our main results. 
We remind the reader that in \eqref{eq:main-prob} the underlying set $X$ need not be $\R^d$ and may be nonconvex; below, when we talk about a first order algorithm for~\eqref{eq:main-prob} we mean an algorithm that can solve~\eqref{eq:main-prob} with access to first-order oracles for $f$. We use $\OPT(f)$ to denote the optimal value of the instance $f$.

\begin{theorem} \label{thm:main}
    Consider an algorithm for~\eqref{eq:main-prob} such that for any instance $f \in \F^0(M,R)$, with access to function values and subgradients of $f$, after $T$ iterations the algorithm reports a feasible solution $x \in X$ with error at most $\err(T, M, R)$, i.e., $f(x) \le \OPT(f) + \err(T, M, R)$.
    
    Then there is an algorithm that, with access to an $\eta$-approximate first-order oracle for $f$ for any $\eta \geq 0$, after $T$ iterations the algorithm returns a feasible solution $\bar{x} \in X$ with value 
    $$f(\bar{x}) \,\le\, \OPT(f) + err(T, M', R) + 4\eta T,$$ where $M' = M + \frac{\eta}{2R}$.
\end{theorem}

    Although we state this theorem as an existence result, our proof is constructive and exactly formulates the desired algorithm via Procedures \ref{proc:optimization} and \ref{proc:lipExt}. Let us illustrate what this theorem says when applied to two classical algorithms for convex optimization (i.e., $X = \R^d$): subgradient methods  and cutting-plane methods~\cite{Nesterov-Book04}. When using \emph{exact} first-order information, the subgradient method produces after $T$ iterations a solution with error at most $O\big(\frac{MR}{\sqrt{T}}\big)$. Applying the procedures mentioned from Theorem~\ref{thm:main} to this algorithm, one obtains an algorithm that uses only $\eta$-approximate first-order information and after $T$ iterations produces a solution whose error is at most $O\Big(\frac{M R + \eta}{\sqrt{T}}\Big) + 2T\eta$. If one can choose the accuracy of the inexact oracle, setting $\eta = \frac{\eps^3}{M^2R^2}$ and $T = \lceil{\frac{M^2R^2}{\eps^2}}\rceil$ gives a solution with error at most $O(\eps)$. Note that this does not involve knowing anything about the original algorithm; it simply illustrates the tradeoff between the oracle accuracy and final solution accuracy.

Similarly, for classical cutting-plane methods (e.g., center-of-gravity, ellipsoid, Vaidya) the error after 
$T$ iterations is at most $O\left(MR\exp\left(\frac{-T}{\poly(d)}\right)\right)$. Thus, with access to 
$\eta$-approximate first-order oracles, we can use our result to produce a solution with error at most $O\left((MR + \frac{\eta}{2})\exp\left(\frac{-T}{\poly(d)}\right)\right) + 4\eta T$. With the desired accuracy of $\eta = O\left(\frac{\eps}{\poly(d)\log\left(\frac{MR}{\eps}\right)}\right)$, and $T = O\left(\poly(d)\log\big(\frac{MR}{\eps}\big)\right)$, it gives a solution with error at most $O(\eps)$.

We next consider the family of {\em $\alpha$-smooth functions}, i.e., the family $\F^1(M,\alpha,R)$ of $M$-Lipschitz convex functions that are differentiable with $\alpha$-Lipschitz gradient maps, whose minimizers are contained in $B(R)$. This is a classical family of objective functions in convex optimization that admits the celebrated accelerated method of~\citet{nesterovAcc} (see \cite{acceleration} for a survey). We give the following universal transfer theorem for algorithms for smooth objective functions. 

\begin{theorem} \label{thm:smooth-main}
    Consider an algorithm for~\eqref{eq:main-prob} such that for any instance in $f \in \F^1(M,\alpha, R)$, with access to function values and subgradients of $f$, after $T$ iterations the algorithm reports a feasible solution $x \in X$ with error at most $\err(T, M, \alpha, R)$, i.e., $f(x) \le \OPT(f) + \err(T, M, \alpha, R)$.
    
    Then for any $\eta \le \frac{\alpha R^2}{5T}$, there is an algorithm that, with access to an $\eta$-approximate first-order oracle for $f$, after $T$ iterations the algorithm returns a feasible solution $\bar{x} \in X$ with value 
    $$f(\bar{x}) \,\le\, \OPT(f) + err(T, M', \alpha', R) +  5 \eta \cdot (T+2),$$ where $M' = M + \frac{\eta}{2R}$, $\alpha' = \alpha \cdot \sqrt{d} \cdot \Big(4 \sqrt{5 (T+1)} + 3 \Big)$.
\end{theorem}

    As an illustration, we apply this transfer theorem to the accelerated algorithm of \citet{nesterovAcc} for continuous optimization ($X = \R^d$): Under \emph{perfect} first-order information, it obtains error $O(\frac{\alpha R^2}{T^2})$ after $T$ iterations. 
    Using our transfer theorem as a wrapper gives an algorithm that, using only $\eta$-approximate first-order information, obtains error $O(\frac{\alpha R^2\sqrt{d}}{T^{1.5}} + \eta T)$; if the accuracy of the oracle is set to $\eta = O(\frac{\alpha R^2 \sqrt{d}}{T^{2.5}})$, this gives an algorithm with error $O(\frac{\alpha R^2\sqrt{d}}{T^{1.5}})$. 
    While this does not recover in full the acceleration of Nesterov's method, the key take away is that a significant amount of acceleration  (i.e., error rates better than those possible for non-smooth functions) can be preserved under inexact oracles in a \emph{universal} way, for \emph{any} accelerated algorithm requiring exact information. 

    \begin{remark} For the sake of exposition, we have assumed that the accuracy $\eta$ of the oracle is fixed and the additional error is $O(\eta T)$. However, one can allow different oracle accuracies $\eta_t$ at each query point $x_t$ and the additional error is $O(\sum_{t}\eta_t)$ (and the parameter $M' = M + \frac{1}{2R}\max_t{\eta_t}$).
    \end{remark}

\subsection{Allowing inexactness in the constraint set}

So far we have assumed that the algorithm has complete knowledge of the constraints $X$. Now, we extend our results to include algorithms that can work with larger classes of constraints that are not fully known up front. In other words, just like the algorithm needs to collect information about $f$, it also needs to collect information about $X$, via another oracle, to be able to solve the problem. To capture the most general algorithms of this type, we formalize this setting by assuming $X$ is of the form $C \cap Z$, where $C$ belongs to a class of closed, convex sets and $Z$ is possibly nonconvex but completely known (e.g., $Z = \Z^{d_1}\times \R^{d_2}$ with $d_1 + d_2 = d$).

\begin{equation}\label{eq:main-prob-flexible}
    \min\{f(x) : x \in C \cap Z\}.
\end{equation}

The algorithm then must collect information about $C$, for which we use the common model of allowing the algorithm access to a separation oracle. Upon receiving a query point $x$, a separation oracle either reports correctly that $x$ is inside $C$ or otherwise returns a separating hyperplane that separates $x$ from $C$. We note that a separation oracle for $C$ is in some sense comparable to a first-order oracle for a convex function $f$; since the pair $(f(x), \nabla f (x))$ can be viewed as providing a supporting hyperplane for the epigraph of $f$ at $x$, using an oracle that returns separating hyperplanes for $C$ provides a comparable way of collecting information about the constraints.  

Let us first precisely define the inexact version of a separation oracle.

\begin{definition} \label{def:approxSep}
  For a closed, convex set $C\subseteq B(R)$ and a query point $\bar{x}\in B(R)$, an \emph{$\eta$-approximate separation oracle} reports a separation response $({flag}, \tilde{g}) \in \{\feas, \infeas\}\times \R^d$ such that if $\bar{x}\in C$ then ${flag} = \feas$ (with no requirement on $\tilde{g}$), and otherwise ${flag} = \infeas$ and $\tilde{g}$ is a unit vector such that there exists some unit vector $g$ satisfying $\ip{g}{x} \leq \ip{g}{\bar{x}}$ for all $x\in C$ and $\norm{\tilde{g} - g }_2\leq \frac{\eta}{4R}$. Given such a $\tilde{g}$ (for $\bar{x} \notin C$), we call the hyperplane through $\bar{x}$ induced by this normal vector an \emph{$\eta$-approximate separating hyperplane} for $\bar{x}$. 
    \end{definition}

We now state our results for algorithms that work with separation oracles. Note that for this, instances of~\eqref{eq:main-prob} have to specify both $f$ and $C$, as opposed to just $f$, since only $Z$ is known but not $C$. We use $\I(M,R, \rho)$ to denote the set of all instances $(f, C)$ where $f:\R^d \to \R$ is an $M$-Lipschitz convex function and $C$ is a compact, convex set that contains a ball of radius $\rho$ and is contained in $B(R)$. We use $OPT(f,C)$ to denote the minimum value of~\eqref{eq:main-prob-flexible}. The ``strict feasibility" assumption of $C$ containing a $\rho$-ball is standard in convex optimization with constraints given via separation oracles. Otherwise, it can be shown that no algorithm will be able to find even an approximately feasible point in a finite number of steps~\cite{Nesterov-Book04}. The first result we state is for pure convex problems, i.e., $Z = \R^d$.

\begin{theorem} \label{thm:approxSep}
    Consider an algorithm for~\eqref{eq:main-prob-flexible} with $Z = \R^d$, such that for any instance in $(f,C) \in \cI(M,R,\rho)$, with access to function values and subgradients of $f$ and separating hyperplanes for $C$, after $T$ iterations the algorithm reports a feasible solution $x \in C$ with error at most $\err(T, M, R, \rho)$, i.e., $f(x) \le \OPT(f,C) + \err(T, M, R, \rho)$.
    
    Then there is an algorithm that, with access to an $\eta_f$-approximate first-order oracle for $f$ and an $\eta_C$-approximate separation oracle for $C$ for any $\eta_f \geq 0$ and $0 \leq \eta_C \le \rho$, after $T$ iterations the algorithm returns a feasible solution $\bar{x} \in C$ with value 
    $$f(\bar{x}) \,\le\, \OPT(f,C) + err(T, M', R, \rho') + 4\eta_f T +  \frac{2  \eta_C MR }{\rho},$$ where $M' = M + \frac{\eta_f}{2R}$ and $\rho' = \rho - \eta_C$.
\end{theorem}

We can handle more general, nonconvex $Z$ with separation oracles under a slightly stronger ``strict feasibility" assumption on $C$: let $\cI^\star(M,R,\rho)$ denote the subclass of instances from $\cI(M,R,\rho)$ where the minimizer $x^\star$ of~\eqref{eq:main-prob-flexible} is $\rho$-deep inside $C$, i.e., $B(x^\star, \rho) \subseteq C$.

\begin{theorem} \label{thm:approxSep-general}
    Consider an algorithm for~\eqref{eq:main-prob-flexible}, such that for any instance in $(f,C) \in \cI^\star(M,R,\rho)$, with access to function values and subgradients of $f$ and separating hyperplanes for $C$, after $T$ iterations the algorithm reports a feasible solution $x \in C\cap Z$ with error at most $\err(T, M, R, \rho)$, i.e., $f(x) \le \OPT(f,C) + \err(T, M, R, \rho)$.
    
    Then there is an algorithm that, with access to an $\eta_f$-approximate first-order oracle for $f$ and an $\eta_C$-approximate separation oracle for $C$ for any $\eta_f \geq 0$ and $0 \leq \eta_C \le \rho$, after $T$ iterations the algorithm returns a feasible solution $\bar{x} \in C\cap Z$ with value 
    $$f(\bar{x}) \,\le\, \OPT(f,C) + err(T, M', R, \rho') + 4\eta_f T,$$ where $M' = M + \frac{\eta_f}{2R}$ and $\rho' = \rho - \eta_C$.
\end{theorem}

\begin{remark}
The objective functions in the above results were allowed to be any $M$-Lipschitz, possibly nondifferentiable, convex function. One can state versions of these results for algorithms that work for the smaller class of $\alpha$-smooth functions (e.g., accelerated projected gradient methods), just as Theorem~\ref{thm:smooth-main} is a version of Theorem~\ref{thm:main} for $\alpha$-smooth objectives. The reason is that the analysis for handling constraints is independent of the arguments needed to handle the objective using inexact oracles; however, for space constraints, we leave the details out of this manuscript. Additionally, one can prove versions of all our theorems for strongly convex objective functions, but we leave these out of the manuscript as well to convey the main message of the paper more crisply.
\end{remark}

\subsection{Relation to existing work} Previous work on inexact first-order information focused on how certain known algorithms perform or can be made to perform under inexact information, most recently on (accelerated) proximal-gradient methods. For instance,~\cite{Devolder_Nesterov2014_First_order_inexact_oracles} analyze the performance of (accelerated) gradient descent in the presence of inexact oracles, with no change to algorithm. They show that simple gradient descent (for unconstrained problems) will return a solution with additional error $O(\eta)$ and accelerated gradient descent incurs an additional error of $O(\eta T)$ (similar to our guarantees). We provide a more thorough comparison of our setting and results with those of \cite{Devolder_Nesterov2014_First_order_inexact_oracles} in Appendix~\ref{sec: comparison with devolder et al}.
    
Similarly,~\cite{schmidt2011convergence} does an analysis for (accelerated) proximal gradient methods, with more complicated forms of the additional error, depending on how well the proximal problems are solved. \cite{Gasnikov2019FastGD} and \cite{cohen2018acceleration} also study gradient methods in inexact settings, with their analyses being specific to particular algorithms. 
     
    In contrast, our result does not assume any knowledge of the internal logic of the algorithm. We must, therefore, use the algorithm in a ``black-box'' manner. We are able to do this by using the inexact oracles to construct a modified instance whose optimal solution is similar in quality to that of the true instance, and where this inexact information from the true instance can be interpreted as {\em exact} information for the modified instance. Thus, we can effectively run the algorithm as a black-box on this modified instance and leverage its error guarantee. Constructing this modified instance in an online fashion requires technical ideas that are new, to the best of our knowledge, in this literature. For instance, it is not even true that given approximate function values and subgradients of a convex function, we can find another convex function that has these as exact function values and subgradients; see Figure \ref{fig:approx-info-nonconvex}. Thus, one cannot directly use the inexact information as is (contrary to what is done in many of the papers dealing with inexact information for {\em specific} algorithms), in the general case we consider. The key is to modify the inexact information so that the information the algorithm receives admits an \emph{extension into a convex function/set} that is still close to the original instance. When dealing with $\alpha$-smooth objectives, the arguments are especially technically challenging since we have to report approximate function and gradient values that allow for a {\em smooth} extension that also approximates the unknown objective well. This involves careful use of new, localized smoothing techniques and maximal couplings of probability distributions. Such smoothing guarantees based on the proximity to the class of smooth functions may be of independent interest (see Theorem~\ref{thm:smoothExt}).

\par{\em New applications:} Since our results apply to algorithms for any ground set $X$, we are able to handle {\em mixed-integer convex optimization}, i.e., $X = C \cap(\Z^{d_1} \times \R^{d_2})$, with inexact oracles. Recently, there have been several applications of such optimization problems in machine learning and statistics~\cite{bertsimas2016best,mazumder2017discrete,bandi2019learning,dedieu2021learning,dey2022using,hazimeh2022sparse,hazimeh2023grouped}. General algorithms for mixed-integer convex optimization, as well as specialized ones designed for specific applications in the above papers, all involve a sophisticated combination of techniques like branch-and-bound, cutting planes and other heuristics. To the best of our knowledge, the performance of these algorithms has never been analyzed under the presence of inexact oracles which can cause issues for all of these components of the algorithm. Our results apply immediately to all these algorithms, precisely because the internal workings of the algorithm are abstracted away in our analysis. This yields the first ever versions of these methods that can work with inexact oracles. Moreover, $X$ can be used to model other types of structured nonconvexities (e.g., complementarity constraints~\cite{cottle2009linear}) and our results show how to adapt algorithms in those settings to work with inexact oracles. Note that this holds for the cases where the convex set $C$ is explicitly known a priori (Theorems~\ref{thm:main} and~\ref{thm:smooth-main}), or must be accessed via separation oracles (Theorems~\ref{thm:approxSep} and~\ref{thm:approxSep-general}).

The remainder of this paper is dedicated to the proof sketches of Theorems~\ref{thm:main} and~\ref{thm:smooth-main}. The missing details and proofs of Theorems~\ref{thm:approxSep} and~\ref{thm:approxSep-general} can be found in the Appendix.

\section{Universal transfer for Lipschitz functions}
\label{sec:optimization}



    In this section we prove our transfer result stated in Theorem~\ref{thm:main}. The proof relies on the following key concept: Given a set of points $x_1,\ldots,x_T \in B(R)$ (e.g., queries made by an optimization algorithm), we say that the sequence of first-order pairs\footnote{We use \emph{first-order pair} as just a more ``visual'' name for a pair in $\R \times \R^d$.} $(f_1,g_1),\ldots,(f_T,g_T) \in \R \times \R^d$ has an \emph{$M$-Lipschitz convex extension}, or simply {\em $M$-extension}, if there is a function $F$ that is convex, $M$-Lipschitz, and such that $f_t = F(x_t)$ and $g_t \in \partial F(x_t)$ for all $t$, i.e., the first-order information of $F$ at the queried points is \textbf{exactly} $\{(f_t,g_t)\}_t$. 
	
	As mentioned in the introduction, the main idea is to feed to the convex optimization algorithm $\cA$ a sequence of pairs $(f_t,g_t)$'s that have an $M$-Lipschitz extension $F$ that is close to the original function $f$. Since the information is consistent with what the algorithm expects when interacting exactly with the function $F$, it will approximately optimize the latter which will then give an approximately optimal solution to the neighboring function $f$. 
	
	Unfortunately, it is easy to see approximate first-order information from $f$ for the queried points $x_t$'s does not necessarily have a Lipschitz convex extension (see Figure \ref{fig:approx-info-nonconvex}). Thus, the main subroutine of our algorithm \emph{Approximate-to-Exact} given below is that given an approximate first-order oracle for $f$, it constructs first-order pairs $(\hat{f}_t, \hat{g}_t)$'s in an online fashion (i.e. $(\hat{f}_t, \hat{g}_t)$ only depends on $x_1,\ldots,x_t$) with the desired extension properties. For a function $g$, let $\|g\|_{\infty} := \sup_x |g(x)|$ denote its sup-norm.

\begin{figure}[htbp]
\begin{center}\includegraphics[width=0.45\textwidth]{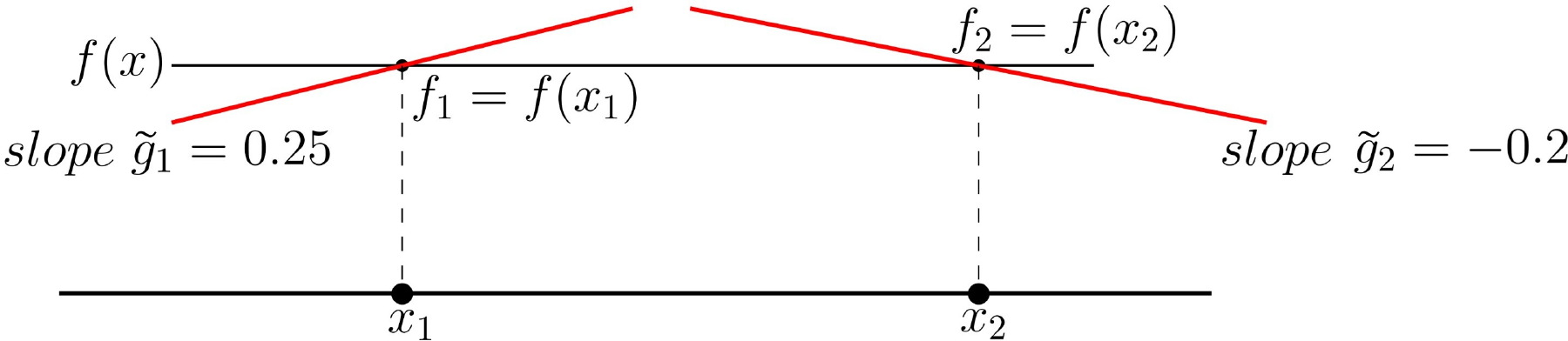}\end{center}
\caption{An example where two approximate function values and subgradients do not have a convex extension. The true function $f$ is constant. The function values are reported with no error. The reported slopes are shown in red. However, these slopes decrease going from $x_1$ to $x_2$ thus eliminating the possibility of any convex function having these values and slopes at $x_1$ and $x_2$.}\label{fig:approx-info-nonconvex}
\end{figure}
 
	\begin{theorem}[Online first-order Lipschitz-extensibility] \label{thm:LipExt}
		Consider an  $M$-Lipschitz convex function $f : B(R) \rightarrow \R$, and a sequence of points $x_1,\ldots,x_T \in B(R)$. There is an online procedure that, given $\eta$-approximate first-order oracle access to $f$, produces first-order pairs $(\hat{f}_1,\hat{g}_1),\ldots,(\hat{f}_T,\hat{g}_T)$ that have a $(M+\frac{\eta}{2 R})$-extension $F : B(R) \rightarrow \R$ satisfying $\|F - f\|_{\infty} \le 2\eta T$. (Moreover, the procedure only probes the approximate oracle at the given points $x_1,\ldots,x_T$.)
	\end{theorem}

    With this at hand, given any first-order algorithm $\cA$ we can run it using only approximate first-order information in the following natural way:
    
    \vspace{4pt}
    \begin{mdframed}
    \begin{proc} \textbf{Approximate-to-Exact}($\cA$, $T$) \label{proc:optimization}

    For each timestep $t = 1,\ldots, T$:

    \vspace{-6pt}
    \begin{enumerate}[leftmargin=18pt]
        \item Receive query point $x_t \in B(R)$ from $\cA$. 
        
        \item \vspace{-6pt} Send point $x_t$ to the $\eta$-approximate oracle for $f$ and receive the information $(\tilde{f}_t, \tilde{g}_t)$.
        
        \item \vspace{-6pt} Use the online procedure from Theorem \ref{thm:LipExt} to construct the first-order pair $(\hat{f}_t, \hat{g}_t)$.
        
        \item \vspace{-6pt} Send $(\hat{f}_t, \hat{g}_t)$ to the algorithm $\cA$.
    \end{enumerate}
    
    \vspace{-6pt} Return the point in $X$ returned by $\cA$.
\end{proc}
\end{mdframed}

 

    \begin{proof}[Proof of Theorem \ref{thm:main}]  Consider a first-order algorithm $\cA$ that, for any $M$-Lipschitz convex function
    , after $T$ iterations returns a point $\bar{x} \in X$ such that $f(\bar{x}) \le \OPT(f) + err(T,M,R)$, 
    We show that running Procedure \ref{proc:optimization} with $\cA$ as input, which only uses an $\eta$-approximate oracle for $f$, returns a point $\bar{x} \in X$ such that $f(\bar{x}) \le \OPT(f) + err(T, M', R) + 4\eta T$ with $M' = M + \frac{\eta}{2R}$.
    
    To see this, let $F$ be an $M'$-extension for the first-order pairs $(\hat{f}_t,\hat{g}_t)$ sent to the algorithm $\cA$ in Procedure \ref{proc:optimization} with  $\|F - f\|_{\infty} \le 2\eta T$, guaranteed 
  by Theorem \ref{thm:LipExt}. This means that the execution of the first-order algorithm $\cA$ during our procedure is exactly the same as executing $\cA$ directly on the convex function $F$. Thus, by the error guarantee of $\cA$, the point $\bar{x} \in X$ returned by $\cA$ after $T$ iterations (which is the same point returned by our procedure) is almost optimal for $F$, i.e., $F(\bar{x}) \le \OPT(F) + \err(T, M', R)$. Since $F$ and $f$ are pointwise within $\pm 2 \eta T$ of each other, the value of the solution $\bar{x}$ with respect to the original function $f$ satisfies
    \begin{align*}
  	f(\bar{x}) \le F(\bar{x}) + 2\eta T &\le \OPT(F) + err(T, M', R) + 2\eta T \\
   &\le \OPT(f) + err(T, M', R) + 4 \eta T,
  \end{align*}
  which proves the desired result. 
\end{proof}
	

	\subsection{Computing Lipschitz-extensible first-order pairs} \label{sec:lipExt}

	In this section we describe the procedure that constructs the first-order pairs with a Lipschitz convex extension $F$ that satisfies $\|F - f\|_{\infty} \le 2\eta T$, proving Theorem \ref{thm:LipExt}. Before getting into the heart of the matter, we show that the latter property can be significantly weakened: instead of requiring both $f(x) \ge F(x) - 2\eta T$ and $f(x) \le F(x) + 2\eta T$ for all $x \in B(R)$, we can relax the latter to only hold for the queried points $x_1,\ldots,x_T$. 
	
    \begin{lemma} \label{lemma:finalFunc}
     	Consider a sequence of points $x_1,\ldots,x_T \in B(R)$, and a sequence of first-order pairs $(\hat{f}_1,\hat{g}_1),\ldots,(\hat{f}_T,\hat{g}_T)$. Consider $\delta > 0$ and $M' \ge M$, and suppose that there is an $M'$-extension $F$ of these first-order pairs that satisfies:
			\begin{align}
				&\underbrace{f(x) \ge F(x) - \delta}_{\textrm{approx. under-approximation}}, ~~~~~~~\forall x \in B(R) \label{eq:underApprox}\\
				&\underbrace{f(x_t) \le F(x_t) + \delta}_{\textrm{approx. queried values}}, ~~~~~\forall t \in  \{1,\ldots,T\}.   \label{eq:queriedApprox}
			\end{align}
	Then the first-order pairs have an $M'$-extension $F'$ such that $\|F' - f\|_{\infty} \le \delta$. In particular, setting $$F'(x) = \max\{F(x), f(x) - \delta\}$$ provides such an extension. 
   \end{lemma} 
  \begin{proof}
  	Define the function $F'$ as $F'(x) := \max\{f(x)- \delta , F(x)\}$ by taking the maximum between $F$ and a downward-shifted $f$. We will show that this function is the desired convex extension of the first-order pairs $(\hat{f}_1,\hat{g}_1),\ldots,(\hat{f}_T,\hat{g}_T)$. 
		 
		 First, to show $\norm{F'-f}_\infty \leq \delta$, by the definition of $F'$ one has $F'(x) \geq f(x) - \delta$ for all $x$. Furthermore, because of the guarantee that $F(x) \leq f(x) + \delta$, we also have $F'(x) \leq f(x)+\delta$ for all $x$; together these imply that $\norm{F'-f}_\infty \leq \delta$. Since $f$ and $F$ are $M'$-Lipschitz convex functions, so is $F'$.
		 
		  It remains to be shown that $F'$ is an extension of the first-order pairs, that is, to show $\hat{f}_t = F'(x_t)$ and $\hat{g}_t \in \partial F'(x_t)$ for all $t = 1, ..., T$. Given property \eqref{eq:queriedApprox}, we have $F(x_t) \ge f(x_t) -  \delta$, and so $F'(x_t) = F(x_t) = \hat{f}_t$. The fact that $F'(x_t) = F(x_t)$ also implies that every vector in $\partial F(x_t)$ is a subgradient of $F'$ at $x_t$, namely $\partial F'(x_t) \supseteq \partial F(x_t) \ni \hat{g}_t$. To see this, recall that since $F$ is convex,  for $\hat{g}_t\in \partial F(x_t)$ we have $F(x_t) + \ip{\hat{g}_t}{x_t - x}\leq F(x)$. Using the fact that $F'(x_t) = F(x_t)$, we thus have $F'(x_t) + \ip{\hat{g}_t}{x_t - x}\leq F(x) \leq F'(x)$ for all $x$, and so any $\hat{g}_t\in \partial F(x_t)$ is also a subgradient for $F'$ at $x_t$, as desired to conclude the proof.  
  \end{proof}
	
	Given Lemma~\ref{lemma:finalFunc}, to prove Theorem \ref{thm:LipExt} it suffices to do the following. Consider a sequence of points $x_1,\ldots,x_T \in B(R)$. Using an $\eta$-approximate first-order oracle to access the function $f$ (at the points $x_1,\ldots,x_T$), we need to produce a sequence of first-order pairs $(\hat{f}_1,\hat{g}_1),\ldots,(\hat{f}_T,\hat{g}_T)$ in an online fashion that have an $M$-extension $F$ achieving the approximations \eqref{eq:underApprox} and \eqref{eq:queriedApprox}. We do this as follows. 
	
	At iteration $t$ we maintain the function $F_t(x) := \max\{\hat{f}_{\tau} + \ip{\hat{g}_{\tau}}{x - x_{\tau}} \,:\, \tau \le t\}$, that is, the maximum of the linear functions induced by the first-order pairs $(\hat{f}_{\tau}, \hat{g}_{\tau})$ constructed up to this point. We would like to define the pairs $(\hat{f}_{\tau}, \hat{g}_{\tau})$ to guarantee that for all $t$, $F_t$ is an $M'$-extension for these pairs, and satisfies \eqref{eq:underApprox} and \eqref{eq:queriedApprox} for $x_1,\ldots,x_t$. In this case, $F = F_T$ gives the desired function. 
	
	For that, suppose the above holds for $t-1$; we will show how to define $(\hat{f}_t,\hat{g}_t)$ to maintain this invariant for $t$. We should think of constructing $F_t$ by taking the maximum of $F_{t-1}$ and a new linear function $\hat{f}_t + \ip{\hat{g}_t}{x - x_t}$. To ensure that $F_t$ is an extension of the first-order pairs thus far, we need to make sure that:
	
	\begin{enumerate}
		\item 
		$\hat{f}_t \ge F_{t-1}(x_t).$
		This is necessary to ensure that $F_t(x_t) = \hat{f}_t$, and also guarantees $\hat{g}_t \in \partial F_t(x_t)$. \vspace{4pt}
		
		\item 
  $\hat{f}_t + \ip{\hat{g}_t}{x_{\tau} - x_t} \le F_{t-1}(x_\tau), ~~~~~\forall \tau \le t-1.
		$	
		This is necessary to ensure that $F_t(x_\tau) = F_{t-1}(x_\tau) = \hat{f}_{\tau}$, and also guarantees $\partial F_t(x_\tau) \supseteq \partial F_{t-1}(x_\tau) \ni \hat{g}_{\tau}$.
	\end{enumerate}
	
	To construct $(\hat{f}_t,\hat{g}_t)$ with these properties, we probe the approximate first-order oracle for $f$ at $x_t$, and receive an answer $(\tilde{f}_t,\tilde{g}_t)$. If setting $(\hat{f}_t,\hat{g}_t) = (\tilde{f}_t,\tilde{g}_t)$ violates the first item above, we simply use the first-order information of $F_{t-1}$ at $x_t$, i.e., we set $\hat{f}_t = F_{t-1}(x_t)$ and $\hat{g}_t \in \partial F_{t-1}(x_t)$. 
	
	If the second item above is violated instead, we shift the value $\tilde{f}_t$ down as little as possible to ensure the desired property, i.e., we set $(\hat{f}_t, \hat{g}_t) = (\tilde{f}_t - s^*, \tilde{g}_t)$ for appropriate $s^* > 0$. With this shifted value, the first item may now be violated, in which case we again just use the current first-order information of $F_{t-1}$. 
	
	These steps are formalized in the following procedure.

%
%
%
%
%

	\begin{mdframed}
    \begin{proc}\label{proc:lipExt} $\phantom{x}$

    Set $F_0(x) = -\infty$. 
    
    For each $t=1,\ldots,T$:

    \vspace{-6pt}
    \begin{enumerate}[leftmargin=18pt]    
        \item \vspace{-4pt} Query the $\eta$-approximate oracle for $f$ at $x_t$, receiving the first-order pair $(\tilde{f}_t, \tilde{g}_t)$. 
        
        \item \vspace{-6pt} Let $s^* := \min\{ s \ge 0 : \tilde{f}_t - s + \ip{\tilde{g}_t}{x_{\tau} - x_t} \le F_{t-1}(x_\tau), \forall \tau \le t-1\}.$ 

        \item
            Let $F_t(x) = \max\{ F_{t - 1}(x), \tilde{f}_t + \langle \tilde{g}_t, x - x_t \rangle - s^* \}$, and then set $\hat{f}_t = F_t(x_t)$, $\hat{g}_t \in \partial F_t(x_t)$.
        
        
        
    \end{enumerate}
		\end{proc}
	\end{mdframed}
    We remark that this requires storing historical values of $\tilde{f}_t$ and $\tilde{g}_t$ (this seems unavoidable to ensure convexity of $F_t$). In terms of computational complexity, we remark that the procedure takes a total of $O(T^2 d)$ operations. We now prove that the functions $F_t$'s have the desired properties. 
 

	\begin{lemma}\label{lemma:FtExt}
    For every $t = 1,\ldots,T$, the function $F_t$ is an $(M+\frac{\eta}{2 R})$-extension of the first-order information pairs $(\hat{f}_1,\hat{g}_1),\ldots,(\hat{f}_t,\hat{g}_t)$.
	\end{lemma}
	
	\begin{proof}
		Since $F_t$ is the maximum over affine functions, it is convex. Moreover, all of its subgradients come from the set $\{\tilde{g}_\tau\}_\tau$, and by the approximation guarantee of the oracle we have that for some subgradient $g \in \partial f(x_\tau)$, $\|\tilde{g}\|_2 \le \|\tilde{g} - g\|_2 + \|g\|_2 \le \frac{\eta}{2 R} + M$, where we used that fact that $f$ is $M$-Lipschitz; thus, $F_t$ is $(M+\frac{\eta}{2 R})$-Lipschitz. 
		
		We prove by induction on $t$ that $F_t$ is an extension of the desired pairs (the base case $t=1$ can be readily verified). Recall $F_t(x_t) = \max\{F_{t-1}(x), H(x)\}$, where $H(x) := \hat{f}_t + \ip{\hat{g}}{x - x_t}$. By the definition of $s^*$, for all $x = x_1,\ldots,x_{t-1}$, this maximum is achieved by the function $F_{t-1}$, giving, by induction, that for all $\tau \le t-1$, $F_t(x_{\tau}) = F_{t-1}(x_\tau) = \hat{f}_\tau$; this also implies that for such $\tau$'s, $\partial F_t(x_\tau) \supseteq \partial F_{t-1}(x_\tau) \ni \hat{g}_{\tau}$, the last inclusion again following by induction. These give the extension property for the pairs $(\hat{f}_\tau, \hat{g}_{\tau})$ with $\tau \le t-1$. 
		
		It remains to verify that this also holds for $\tau = t$. Now the maximum in the definition of $F_t(x_t)$ is achieved by the function $H$: if $\tilde{f}_t - s^* \ge F_{t-1}(x_t)$, the procedure sets $\hat{f}_t = \tilde{f}_t -s^*$ and we have $H(x_t) = \hat{f}_t \ge F_{t-1}(x_t)$; otherwise the procedure sets $\hat{f}_t = F_{t-1}(x_t)$ and so $H(x_t) = F_{t-1}(x_t)$. Again this implies that $\partial F_t(x_t) \supseteq \partial H(x_t) = \{\hat{g}_t\}$. This concludes the proof of the lemma. 		
	\end{proof}

	Finally, we show that the functions $F_t$ approximate $f$ according to \eqref{eq:underApprox} and \eqref{eq:queriedApprox}. 
	
	\begin{lemma} \label{lemma:FtApprox}
		For every $t = 1,\ldots,T$, the $F_t$ satisfies inequalities \eqref{eq:underApprox} and \eqref{eq:queriedApprox} with $\delta = 2 \eta t$. 
	\end{lemma}
	
	\begin{proof}
		Again we prove this by induction on $t$. Fix $t$. Let $\Delta := \tilde{f}_t - f(x_t)$ be the error the inexact oracle makes on the function value. We claim that the shift $s^*$ used in iteration $t$ of Procedure~\ref{proc:lipExt} satisfies $s^* \le \max\{0, \Delta + 2 \eta t \}$. To see this, the $\eta$-approximation of the oracle guarantees that there is a subgradient $g \in \partial f(x_t)$ such that $\|\tilde{g}_t - g\|_2 \le \frac{\eta}{2R}$, and so for every $\tau \le t-1$
		\begin{align}
			&\tilde{f}_t + \ip{\tilde{g}_t}{x_\tau - x_t}\nonumber\\
   =& 	\Delta + \underbrace{f(x_t) + \ip{g}{x_\tau - x_t}}_{\le f(x_\tau)} + \underbrace{\ip{\tilde{g}_t - g}{x_\tau - x_t}}_{\le \|\tilde{g}_t - g\|_2 \|x_\tau - x_t\|_2 \le \eta}\nonumber\\ 
   \le& F_{t-1}(x_\tau) + \Delta + 2 t \eta, \label{eq:FTApprox1}
		\end{align}
		the first underbrace following since $g$ is a subgradient of $f$, and the last inequality following from the induction hypothesis $F_{t-1}(x_\tau) \ge f(x_\tau) - 2 (t-1) \eta$ (inequality \eqref{eq:queriedApprox}); the optimality of $s^*$ then guarantees that it is at most $\max\{0, \Delta + 2\eta t\}$, proving the claim. 
				
		Now we show that $F_t$ satisfies the desired bounds, namely $F_t(x_\tau) \ge f(x_\tau) - 2 \eta t$ for all $\tau \le t$, and $F_t(x) \le f(x) + 2 \eta t$ for all $x \in B(R)$. 
		From the inductive hypothesis, for $\tau \le t-1$ we have $F_t(x_\tau) \ge F_{t-1}(x_\tau) \ge f(x_\tau) - 2\eta (t-1)$, giving the first bound for these $x_\tau$. For $x_t$, notice that $F_t(x_t) \ge \tilde{f}_t - s^*$.
  Therefore,   \begin{align*}
       F_t(x_t) &\ge \tilde{f}_t - s^* \ge \tilde{f}_t - \max\{0, \Delta - 2\eta t\} \\
       &\ge \max\{f(x_t) - \eta\,,\, f(x_t) - 2 \eta t\} = f(x_t) - 2 \eta t,
  \end{align*}
  where in the second inequality we used the upper bound on the shift $s^* \le \max\{0, \Delta + 2\eta t\}$, and in the next inequality we used the guarantee $|\tilde{f}_t - f(x_t)| \le \eta$ from the approximate oracle.
		
		For the upper bound $F_t(x) \le f(x) + 2 \eta t$, by the inductive hypothesis $F_{t-1}(x) \le f(x) + 2 \eta (t-1)$. Moreover, the same development as in \eqref{eq:FTApprox1} reveals that  \begin{align*}
      \tilde{f}_t + \ip{\tilde{g}_t}{x - x_t} - s^* &\le \tilde{f}_t + \ip{\tilde{g}_t}{x - x_t}\\
      &\le f(x) + \Delta + \eta \le f(x) + 2 \eta,
  \end{align*}
  where the last inequality again uses that $\Delta = \tilde{f}_t - f(x_t) \le \eta$ due to the guarantee of the approximate oracle. Thus, $F_t(x) \le \max\{f(x) + 2\eta (t-1), f(x) + 2\eta \} \le f(x) + 2 \eta t$, giving the desired bound. This concludes the proof of the lemma. 		
  
	\end{proof}
	
	Combining Lemmas \ref{lemma:finalFunc}, \ref{lemma:FtExt}, and \ref{lemma:FtApprox} shows that the first-order pairs produced by Procedure \ref{proc:lipExt} satisfies the properties stated in Theorem \ref{thm:LipExt}, finally concluding its proof. 
	

\section{Universal transfer for smooth functions} \label{sec:transfSmooth}

    
    In this section we prove our transfer theorem for smooth functions stated in Theorem \ref{thm:smooth-main}. Recall that a function $f$ is $\alpha$-smooth if it has $\alpha$-Lipschitz gradients:
\begin{align*}
    \forall x, y \in \mathbb{R}^d, \quad \|\nabla f(x) - \nabla f(y)\| \leq \alpha \|x - y\|.
\end{align*}

    As in the proof of the previous transfer theorem, the core element is the following: Given the sequence of iterates $x_1, \ldots, x_t$ of a black-box optimization algorithm and access to an approximate first-order oracle to the smooth objective function $f$, construct in an online fashion first-order pairs $(\hat{f}_t, \hat{g}_t)$ and, implicitly, a \emph{smooth} function $S$ close to the original $f$ such that $(\hat{f}_t, \hat{g}_t)$ provide exactly the value and gradient of $S$ at $x_t$.
    
	\begin{theorem}[Online first-order smooth-extensibility] \label{thm:smoothExt}
		Consider an $\alpha$-smooth, $M$-Lipschitz convex function $f : \R^d \rightarrow \R$, and a sequence of points $x_1,\ldots,x_T \in B(R)$. Then, for $\eta \le \frac{\alpha R^2}{5T}$, there is an online procedure that given $\eta$-approximate first-order oracle access to $f$, produces first-order pairs $(\hat{f}_1,\hat{g}_1),\ldots,(\hat{f}_T,\hat{g}_T)$ that have an $\alpha'$-smooth $(M+\frac{\eta}{2 R})$-extension $S : B(R) \rightarrow \R$ satisfying $\|S - f\|_{\infty} \le 5 \eta \, (T+2)$, where $\alpha' = \alpha \, \sqrt{d} \, \Big(4 \sqrt{5 \cdot (T+1)} + 3\Big)$. Moreover, the procedure only probes the approximate oracle at the given points $x_1,\ldots,x_T$.
	\end{theorem}

    In the previous section, the extension was created by adding a new linear function at every iteration; this produced the piecewise linear (non-smooth) functions $F_t$ in the previous section. Having to construct a \emph{smooth} extension creates a challenge. Our approach is to apply a \emph{smoothing} procedure to these piecewise linear functions, in an online manner. One issue is that most standard smoothing procedures (e.g., via inf-convolution~\cite{beckTeboulle} or Gaussian smoothing~\cite{Nesterov2005_Smooth_minimization_of_nonsmooth}) may use the values of the non-smooth base function over the whole domain; in our online construction, at a given point in time we have determined the value of the function only in a neighborhood of the previous iterates, and the updated functions can change at points outside these small neighborhoods. Thus, we employ a localized smoothing procedure. Moreover, we need the procedure to leverage the fact that the non-smooth base function is close to a smooth one, and produce stronger smoothing guarantees by making use thereof. 
    We start by describing this smoothing technique and its properties, and then describe the full procedure that gives Theorem \ref{thm:smoothExt}. 
    

    \paragraph{Randomized smoothing of almost smooth functions.} Given a function $h : \R^d \rightarrow \R$ and a radius $r > 0$, we define the smoothed function $h_r$ by $h_r(x) :=  \E h(x + r U),$ where $U$ is uniformly distributed on the unit ball $B(1)$. It is well-known that when $h$ is convex and $M$-Lipschitz, then $h_r$ is differentiable, also $M$-Lipschitz, and, most importantly, is $\frac{M \sqrt{d}}{r}$-smooth~\cite{YOUSEFIAN201256}. However, we show that the smoothing parameter can be significantly improved when the function $h$ is already close to a smooth function. The proof is deferred to Appendix \ref{sec:randSmooth}.

    \begin{lemma} \label{lemma:smoothing}
        Let $h : B(4R) \rightarrow \R$ be a convex function such that there exists an $\alpha$-smooth convex function $f : B(4R) \rightarrow \R$ with $\|h-f\|_{\infty} \le \e$, for $\e \le \alpha R^2$. Then, for $r \le R$ the smoothed function $h_r : B(R) \rightarrow \R$ (so restricted to the ball $B(R)$) satisfies:
        
    \vspace{-10pt}
    \begin{enumerate}
    	\item $h_r$ is $\Big(\frac{4 \sqrt{\alpha \e d}}{r} + 3\alpha \sqrt{d}\Big)$-smooth \vspace{-6pt}
    	\item $|h_r(x) - f(x)| \le \e + \frac{\alpha r^2}{2}$ for all $x \in B(R)$.
    \end{enumerate}
    \end{lemma}

	
    \paragraph{Construction of the smooth-extension.} 
    As mentioned, in each iteration $t$ we will maintain a piecewise linear function $F_t$ constructed very similarly to the proof of Theorem~\ref{thm:LipExt}. Now we will also maintain the smoothened version $(F_t)_r$ of this function that uses the randomized smoothing discussed above (for a particular value of $r$). Our transfer procedure then returns the first-order information $\hat{f}_t := (F_t)_r(x_t)$ and $\hat{g}_t := \nabla (F_t)_r(x_t)$ of the latter. The final smooth function $S : B(R) \rightarrow \R$ compatible with the first-order information returned by the procedure will be given, as in Lemma \ref{lemma:finalFunc}, by taking the maximum between the final $F_T$ and a shifted version of the original function $f$.

        The main difference in how the functions $F_t$'s are constructed, compared to the proof of Theorem~\ref{thm:LipExt}, is the following. Previously, in order to ensure that $F_T$ (and so the final extension) was compatible with the first-order pairs output in earlier iterations, we needed to ``protect'' the points $x_t$ and ensure that the function values and gradients at these points did not change over time, e.g., we needed $F_T(x_t) = F_t(x_t)$. But now the first-order pair output for the query point $x_t$ depends not only on the value of $F_t$ at $x_t$, but also on the values on the whole \emph{ball} $B(x_t,r)$ that are used to determine the smoothed function $(F_t)_r$ at $x_t$. Thus, we will now need to ``protect'' these balls and ensure that the function values over them do not change in later iterations. 

        We now formalize the construction of the functions $F_t$, the first-order information returned, and the final extension $S$ in Procedure \ref{proc:smoothExt}.  
        
		\begin{mdframed}
    \begin{proc}\label{proc:smoothExt} $\phantom{x}$

    Set $r = \sqrt{\eta/\alpha}$ and $F_0(x) = -\infty$. 
    
    For each $t=1,\ldots,T$:

    \vspace{-6pt}
    \begin{enumerate}[leftmargin=18pt]    
        \item \vspace{-4pt} Query the $\eta$-approximate oracle for $f$ at $x_t$, receiving the first-order pair $(\tilde{f}_t, \tilde{g}_t)$.   
        
				\item \vspace{-6pt} Define the function $F_t$ by setting $F_t(x) = \max\{F_{t-1}(x)\,,\, \tilde{f}_t + \ip{\tilde{g}_t}{x - x_t} - (4 \eta t + \alpha r^2 t + 2 \eta) \}$ for all $x$       
        
				\item \vspace{-6pt} Output the first-order information of the randomly smoothed function $(F_t)_r$: $\hat{f}_t := (F_t)_r(x_t)$ and $\hat{g}_t := \nabla (F_t)_r(x_t)$
    \end{enumerate}
 
    \vspace{-6pt}    
    Define the function $S : B(R) \rightarrow \R$ by $S = (\max\{F_T, f- 4 \eta (T+1) + \alpha r^2 (T+1)\})_r$, where $\max$ denotes pointwise maximum.
		\end{proc}
	\end{mdframed}
 
The proof that this procedure indeed yields Theorem \ref{thm:smoothExt} is presented in Appendix \ref{app:smoothExt}.

\section*{Acknowledgments}

We would like to thank the reviewers for their detailed and insightful feedback that has corrected inaccuracies in the previous proofs and have improved the presentation of the paper. 

The first and fourth authors would like to acknowledge support from Air Force Office of Scientific Research (AFOSR) grant FA95502010341 and National Science Foundation (NSF) grant CCF2006587. The second author was supported in part by the Coordena\c{c}\~ao de Aperfei\c{c}oamento de Pessoal de N\'ivel Superior (CAPES, Brasil) - Finance Code 001, and by Bolsa de Produtividade em Pesquisa $\#3$12751/2021-4 from CNPq. 

\section*{Impact Statement}
 This paper presents work whose goal is to advance the fields of Machine Learning and Optimization. There are many potential societal consequences of our work, none which we feel must be specifically highlighted here.

\small

{
\bibliographystyle{icml2024}
\bibliography{full-bib}
}

\bigskip
\appendix

{\LARGE \bf Appendix}

\section{Universal transfer for smooth functions}

    In this section we present the missing proofs for our transfer theorem for smooth functions from Section \ref{sec:transfSmooth}. We start by recalling the definition of a smooth function.

\begin{definition}
    A function $f: \mathbb{R}^d \to \mathbb{R}$ is said to be $\alpha$-smooth if it is differentiable and its gradient is Lipschitz continuous with a Lipschitz constant $\alpha$, namely 
\begin{align*}
    \forall x, y \in \mathbb{R}^d, \quad \|\nabla f(x) - \nabla f(y)\| \leq \alpha \|x - y\|.
\end{align*}
\end{definition}

An $\alpha$-smooth function possesses the following useful upper bounding property: for  $x, y \in \R^n$:

\begin{equation}
    f(y) \leq f(x) + \langle \nabla f(x), y - x \rangle + \frac{\alpha}{2} \|y - x\|^2.
\end{equation}


	\subsection{Proof of Lemma \ref{lemma:smoothing}} \label{sec:randSmooth}

    Let $h$ and $f$ be the convex functions over the ball $B(4R)$ satisfying the statement of the lemma, i.e., $\|h - f\|_{\infty} \le \e$ and $f$ is $\alpha$-smooth. Recall that the smoothed function $h_r$ is defined as $h_r(x) = \E h(x + rU)$ for $x \in B(R)$, where $U$ is a random vector uniformly distributed on the unit ball $B(1)$ and $r \le R$.     
    
    The first observation is that since $h$ is close to $f$ and the latter is smooth, their (sub)gradients are close to each other; the same also holds between $h_r$ and $f$.

    \begin{lemma} \label{lemma:gradFG}
        We have the following:
        \begin{enumerate}
        	\item $\|\partial h(x) - \nabla f(x)\| \le 2 \sqrt{\alpha \e}$ for every $x \in B(2R)$ and every subgradient $\partial h(x)$.
        	\item $\|\nabla h_r(x) - \nabla f(x)\| \le 2 \sqrt{\alpha \e} + \alpha r$ for every $x \in B(R)$.
        \end{enumerate} 
    \end{lemma}

    \begin{proof}
    	To prove the first item, fix $x \in B(2R)$ and let $y$ be such that $$x-y = 2 \sqrt{\frac{\e}{\alpha}} \cdot \frac{\partial h(x)-\nabla f(x)}{\|\partial h(x) - \nabla f(x)\|}.$$ Since $x \in B(2R)$, notice that $y$ has norm at most $2R + 2 \sqrt{\e/\alpha}$, which by assumption of $\e$ is at most $4R$; thus, $y$ is in the domain of $h$ and $f$. 
     
        Then using $\alpha$-smoothness of $f$, $\|h -  f\|_\infty \le \e$, and convexity of $h$, we have
     \begin{align*}
            &f(x) + \ip{\nabla f(x)}{y-x} + \frac{\alpha}{2} \|x-y\|^2 \,\ge\, f(y) \,\ge\, h(y) - \e \,\\
            &\qquad\qquad\qquad\qquad\qquad~~\ge\, h(x) + \ip{\partial h(x)}{y-x} - \e,
     \end{align*}
     and so 
     \begin{align*}
            \ip{\partial h(x) - \nabla f(x)}{y-x} &\le  f(x) - h(x) + \e + \frac{\alpha}{2} \|x-y\|^2 \\
            &\le  2\e + \frac{\alpha}{2} \|x-y\|^2\,.
      \end{align*}
      Plugging the definition of $y$ on this expression gives 
      \begin{align*}
        2 \sqrt{\frac{\e}{\alpha}} \cdot \|\partial h(x) - \nabla f(x)\| \le 4\e,
      \end{align*}
      and so $\|\partial h(x) - \nabla f(x)\| \le 2 \sqrt{\alpha \e}$, which gives the first item of the lemma. 
      
      For the second item, again let $U$ be uniformly distributed in $B(1)$. This random variable is sufficiently regular that gradients and expectations commute, namely $\nabla h_r(x) = \nabla (\E\, h(x + rU)) = \E\, \partial h(x + rU)$, were $\partial h$ denotes any subgradient of $h$~\cite{bertsekas}.
      Then applying Jensen's inequality, for any $x \in B(R)$ we get
      \begin{align*}
      \|\nabla h_r(x) - \nabla f(x)\| &= \|\E\, \partial h(x + r U) - \nabla f(x)\| \\
      &\le \E\, \|\partial h(x + r U) - \nabla f(x)\|.
      \end{align*}
      Also, for any unit-norm vector $u$ we have 
      \begin{align*}
      \|\partial h(x + r u) - \nabla f(x) \| &\le \|\partial h(x + r u) - \nabla f(x + r u) \| \\
      &~~~~~~+ \|\nabla f(x + r u) - \nabla f(x) \| \\
      &\le 2 \sqrt{\alpha \e} + \alpha r,
      \end{align*}
   where the last inequality follows from Item 1 of the lemma (since $r \le R$, $x + ru$ has norm at most $R + r \le 2R$ and so the item can indeed be applied) and $\alpha$-smoothness of $f$ (which is equivalent to $\|\nabla f(z) - \nabla f(z')\| \le \alpha \|z-z'\|$~\cite{nesterov2018lectures}). This concludes the proof.  
    \end{proof}

    The second element that we will need is a bound on the total variation between the the uniform distributions on the two same-radius balls with different centers.

    
    \begin{lemma} \label{lemma:totalVarBall}
    Let $X \in \R^d$ be the uniformly distributed on $B(x,r)$ and $Y \in \R^d$ be uniformly distributed on $B(y,r)$. Then there is a random variable $(X',Y') \in \R^{2d}$ where $X'$ has the same distribution as $X$ and $Y'$ the same distribution as $Y$, and where $\Pr(X' \neq Y') \le \frac{\|x-y\| \sqrt{d}}{r}$.   
    \end{lemma}

    \begin{proof}[Proof sketch]
    This folklore result can be obtained as follows. Let $\mu_z$ be the uniform distribution over $B(z,r)$. Since $\mu_x$ and $\mu_y$ are the distribution of $X$ and $Y$, by the Maximal Coupling Lemma (Theorem 5.2 of \cite{lindvall}) there is a random variable $(X',Y') \in \R^{2d}$ where $X' \sim \mu_x$ and $Y' \sim \mu_y$ and $\Pr(X' \neq Y') = \frac{1}{2} \int |\textrm{d}\mu_x(z) - \textrm{d}\mu_y(z)| \textrm{d}z$. Moreover, it is known that the right hand side is at most $\frac{\|x-y\| \sqrt{d}}{r}$, see for example inequality (39) of \cite{yousefianArxiv} (plus the estimate from \cite{wallisIneq}).
    \end{proof}

    We are now ready to prove Lemma \ref{lemma:smoothing}.

    \begin{proof}[Proof of Lemma \ref{lemma:smoothing}]
      Item 1: We prove that $\|\nabla h_r(x) - \nabla h_r(y)\| \le (\frac{4 \sqrt{\alpha \e d}}{r} + 3\alpha \sqrt{d}) \|x-y\|$ for all $x,y$. In fact, it suffices to prove this for $x,y$ where $\|x-y\| \le r$, since the inequality can then be chained to obtain the result for any pair of points. 
      
     Then fix $x,y$ with $\|x-y\| \le r$. Using the notation from Lemma \ref{lemma:totalVarBall}, $\nabla h_r(x) = \E\, \partial h(X')$ and $\nabla h_r(y) = \E\, \partial h(Y')$ and $\Pr(X' \neq Y') \le \frac{\|x-y\| \sqrt{d}}{r}$. Applying Jensen's inequality,
        \begin{align*}
            &\|\nabla h_r(x) - \nabla h_r(y)\| \le \E_{X',Y'} \|\partial h(X') - \partial h(Y')\| \\
            &\le \frac{\|x-y\| \sqrt{d}}{r} \cdot \max_{x', y' \in B(x,r) \cup B(y,r)} \|\partial h(x') - \partial h(y')\|. 
        \end{align*} 
        We upper bound the last term by applying triangle inequality and then Lemma \ref{lemma:gradFG}:
        \begin{align*}
        \|\partial h(x') - \nabla h(y')\| &\le \|\partial h(x') - \nabla f(x')\| \\
        &~~~~~+ \|\nabla f(x') - \nabla f(y')\| \\
        &~~~~~+ \|\nabla f(y') - \partial h(y')\| \\
        &\le 4 \sqrt{\alpha \e} + \alpha \,\|x'-y'\| \\
        &\le 4 \sqrt{\alpha \e} + 3\alpha r,
        \end{align*}
        where the second inequality uses that $f$ is $\alpha$-smooth, and the last inequality uses the assumption $\|x-y\| \le r$. Plugging this into the previous inequality gives $$ \|\nabla h_r(x) - \nabla h_r(y)\| \le \bigg(\frac{4 \sqrt{\alpha \e d}}{r} + 3\alpha \sqrt{d}\bigg) \cdot \|x-y\|,$$ as desired. 
        
        Second item: We now show that $\|h_r - f\|_{\infty} \le \e + \frac{\alpha r^2}{2}$. Fix $x \in \R^d$, and again let $U$ be uniformly distributed in the unit ball. Using the assumption $\|h - f\|_{\infty} \le \e$ and convexity of $f$, we have
    \begin{align*}
    h(x + r U) \ge f(x + r U) - \e \ge f(x) + \ip{\nabla f(x)}{r U} - \e.
    \end{align*}  
    Since $U$ has mean zero, taking expectations gives $h_r(x) \ge f(x) - \e$. Similarly, since $f$ is $\alpha$-smooth
    \begin{align*}
    h(x + r U) &\le f(x + r U) + \e \\
    					 &\le f(x) + \ip{\nabla f(x)}{r U} + \frac{\alpha}{2} \|r U\|^2 + \e,
    \end{align*}  
    and taking expectations gives $h_r(x) \le f(x) + \frac{\alpha r^2}{2} + \e$. Together, these yield $|h_r(x) - f(x)| \le \e + \frac{\alpha r^2}{2}$, thus proving the result. This concludes the proof of the theorem.
    \end{proof}
    
	
	\subsection{Proof of Theorem \ref{thm:smoothExt}} \label{app:smoothExt}	

	Throughout this section, fix an $\alpha$-smooth $M$-Lipschitz function $f : B(R) \rightarrow \R$. Recall that we have a sequence of queried points $x_1,\ldots,x_T \in B(R)$ and access to an $\eta$-approximate first-order oracle for $f$. Our goal is to produce, in an online fashion, a sequence of first-order pairs $(\hat{f}_1,\hat{g}_1),\ldots,(\hat{f}_T,\hat{g}_T)$ for the queried points and a function $S$ that is smooth, Lipschitz, and compatible with these first-order pairs (i.e., $S(x_t) = \hat{f}_t$ and $\nabla S(x_t) = \hat{g}_t$). 

     As mentioned, in each iteration $t$ we will keep a piecewise linear function $F_t$ and their smoothened version $(F_t)_r$ (by using the randomized smoothing from the previous section for a specific value of $r$). Our transfer procedure then returns the first-order information $\hat{f}_t := (F_t)_r(x_t)$ and $\hat{g}_t := \nabla (F_t)_r(x_t)$ of the latter. The final smooth function $S : B(R) \rightarrow \R$ compatible with the first-order information output by the procedure will be given, as in Lemma \ref{lemma:finalFunc}, by using the maximum between the final $F_T$ and a shifted version of the original function $f$. Also recall that in order to ensure the compatibility of $S$ with the  first-order information $(\hat{f}_t, \hat{g}_t)$ output throughout the process, we need to ``protect'' the points $x_t$ and ensure that the function values and gradients at these points did not change across iterations, i.e. $(F_T)_r(x_t) = (F_t)_r(x_t)$ and $\nabla (F_T)_r(x_t) = \nabla (F_t)_r(x_t)$. Since $(F_{t'})_r(x_t)$ depends on the values of $F_{t'}$ at the ball $B(x_t, r)$ around $x_t$, we need to ``protect'' $F_{t'}$ on these balls, namely to have $F_T(x) = F_t(x)$ for all $x \in B(x_t, r)$. 
     
    For convenience, we recall the exact construction of the functions $F_t$, the first-order information returned, and the final extension $S$. In hindsight, set $r := \sqrt{\eta/\alpha}$, and for every $t$ define the shift $s_t := 4 \eta t + \alpha r^2 t$. 
 
		\begin{mdframed}
    \textbf{Procedure \ref{proc:smoothExt}.}

     Set $F_0(x) = -\infty$. 
     
     For each $t=1,\ldots,T$:

    \vspace{-6pt}
    \begin{enumerate}[leftmargin=18pt]    
        \item \vspace{-4pt} Query the $\eta$-approximate oracle for $f$ at $x_t$, receiving the first-order pair $(\tilde{f}_t, \tilde{g}_t)$. 
        
				\item \vspace{-6pt} Define the function $F_t$ by setting $F_t(x) = \max\{F_{t-1}(x)\,,\, \tilde{f}_t + \ip{\tilde{g}_t}{x - x_t} - (s_t + 2 \eta) \}$ for all $x$.       
        
				\item \vspace{-6pt} Output the first-order information of the randomly smoothed function $(F_t)_r$: $\hat{f}_t := (F_t)_r(x_t)$ and $\hat{g}_t := \nabla (F_t)_r(x_t)$.
    \end{enumerate}
 
    \vspace{-6pt}    
    Define the function $S : B(R) \rightarrow \R$ by $S = (\max\{F_T, f- s_{T+1}\})_r$, where $\max$ denotes pointwise maximum.
	\end{mdframed}
	
	We now prove the main properties of the functions $F_t$, formulated in the following lemma. The first two are similar to \eqref{eq:underApprox} and \eqref{eq:queriedApprox} used in our non-smooth transfer result and guarantee, loosely speaking, that $F_t$ is close to the original function $f$. The third property is precisely the ``ball protection'' idea discussed above.

	\begin{lemma} \label{lemma:smoothFt}
	For all $t$, the function $F_t$ satisfies the following:
	
	\vspace{-6pt} 
\begin{enumerate}
    \item $F_t(x) \le f(x)$ for every $x \in B(4R)$

		\item For every $t' \le t$, we have $F_t(x) \ge f(x) - s_{t+1}$ for all $x \in B(x_{t'}, \sqrt{2}\, r)$ 

    \item For every $t' \le t$, we have $F_t(x) = F_{t'}(x)$ for every $x \in B(x_{t'}, \sqrt{2}\, r)$. In particular $(F_t)_r(x_{t'}) = (F_{t'})_r(x_{t'})$ and $\nabla (F_t)_r(x_{t'}) = \nabla (F_{t'})_r(x_{t'})$. 
\end{enumerate}
	\end{lemma}	
	
	\begin{proof}
		We prove these properties by induction on $t$. 
		
		First item: Since the property holds by induction for $F_{t-1}$ and $F_t(x) = \max\{F_{t-1}(x)\,,\, \tilde{f}_t + \ip{\tilde{g}_t}{x - x_t} - (s_t + 2 \eta)\}$, it suffices to show that 
		\begin{align}
		\tilde{f}_t + \ip{\tilde{g}_t}{x - x_t} - (s_t + 2 \eta) \le f(x) \label{eq:item1Ft}
		\end{align}
		for all $x \in B(4R)$. For that, since $(\tilde{f}_t, \tilde{g}_t)$ comes from an $\eta$-approximate first-order oracle, by definition $|\tilde{f}_t - f(x_t)| \le \eta$ and $\|\tilde{g}_t - \nabla f(x_t)\| \le \frac{\eta}{2R}$; in particular, $|\ip{\tilde{g}_t - \nabla f(x_t)}{x-x_t}| \le \|\tilde{g}_t - \nabla f(x_t)\| \|x-x_t\| \le \frac{5\eta}{2}$ for every $x \in B(4R)$ (since also $x_t \in B(R)$, by assumption). Then using convexity of $f$ we get
		\begin{align}
			f(x) &\ge f(x_t) + \ip{\nabla f(x_t)}{x - x_t} \notag\\
			&\ge \tilde{f}_t + \ip{\tilde{g}_t}{x - x_t} - \eta - \frac{5\eta}{2} ,  \label{eq:item1Conv}
		\end{align}
		which implies \eqref{eq:item1Ft} as desired, since $s_t + 2\eta \ge \eta + \frac{5\eta}{2}$. 
		
		Second item: Again since this property holds by induction for $F_{t-1}$, it suffices to show 
		\begin{align}
		\tilde{f}_t + \ip{\tilde{g}_t}{x - x_t} - (s_t + 2\eta) \ge f(x) - s_{t+1} \label{eq:item2Ft}
		\end{align}
		for all $x \in B(x_t, \sqrt{2}\,r)$. Since $f$ is $\alpha$-smooth, for every such $x$ we have 
		\begin{align}
			f(x) &\le f(x_t) + \ip{\nabla f(x_t)}{x - x_t} + \frac{\alpha}{2} \|x - x_t\|^2 \notag\\
			&\le \tilde{f}_t + \ip{\tilde{g}_t}{x - x_t} + 2 \eta + \alpha r^2. \label{eq:item2Smooth}
		\end{align}
		Since $s_{t+1} = s_t + 4 \eta + \alpha r^2$, reorganizing the terms gives \eqref{eq:item2Ft} as desired.  
		
		Third item: To show that for every $t' \le t$, we have $F_t(x) = F_{t'}(x)$ for every $x \in B(x_{t'}, \sqrt{2}\, r)$, it suffices to show that for every $t' < t$
		\begin{align}
		\tilde{f}_t + \ip{\tilde{g}_t}{x - x_t} - (s_t + 2\eta) \le F_{t-1}(x) \label{eq:item3Ft}
		\end{align}		
		for all $x \in B(x_{t'}, \sqrt{2}\,r)$. For that, first notice that for all $t' < t$ we have $F_{t-1} \ge F_{t'}$, and the latter can be lower bounded by the affine term added during iteration $t'$. Combining this with \eqref{eq:item2Smooth}, applied to iteration $t'$, we get for all $x \in B(x_{t'}, \sqrt{2}\,r)$
		\begin{align*}
			F_{t-1}(x) &\ge \tilde{f}_{t'} + \ip{\tilde{g}_{t'}}{x - x_{t'}} - (s_{t'} + 2 \eta)\\
			&\ge f(x) - (s_{t'} + 4\eta + \alpha r^2)\\
			&\ge \tilde{f}_t + \ip{\tilde{g}_t}{x - x_t} - (s_{t'} + 6 \eta + \alpha r^2),
		\end{align*}
		where the last inequality uses \eqref{eq:item1Conv}. Since $s_t \ge s_{t'} + 4 \eta + \alpha r^2$, this implies \eqref{eq:item3Ft} as desired. 
		
		To conclude the proof of this item, notice that $(F_t)_r(x_{t'})$ (respectively $(F_{t'})_r(x_{t'})$) only depends on the values of $F_t$ (resp. $F_{t'}$) on the ball $B(x_{t'},r)$. Since we just showed the value of $F_t$ and $F_{t'}$ agree on this ball, we get $(F_t)_r(x_{t'}) = (F_{t'})_r(x_{t'})$. Similarly, the gradient $\nabla (F_t)_r(x_{t'})$ only depends on the values of $F_t$ on an arbitrarily small open neighborhood of the ball $B(x_{t'}, r)$, and the same holds for $\nabla (F_{t'})_r(x_{t'})$. Since the bigger ball $B(x_{t'}, \sqrt{2}\,r)$ contains such a neighborhood, we again obtain $\nabla (F_t)_r(x_{t'}) = \nabla (F_{t'})_r(x_{t'})$. This concludes the proof of the lemma.  
	\end{proof}

	We are now ready to prove Theorem \ref{thm:smoothExt}. 
	
	\begin{proof}[Proof of Theorem \ref{thm:smoothExt}]
	We need to prove that the function $S : B(R) \rightarrow \R$ defined in Procedure \ref{proc:smoothExt} satisfies: 
	
	\vspace{-6pt}
	\begin{enumerate}
		\item $\|S - f\|_{\infty} \le s_{T+1} + \frac{\alpha r^2}{2}$.
		\item $S$ is $\Big(\frac{4 \sqrt{\alpha \cdot d \cdot s_{T+1}}}{r} + 3\alpha \sqrt{d}\Big)$-smooth
		\item $S$ is $(M + \frac{\eta}{2 R})$-Lipschitz 
		\item $S$ is an extension for the first-order pairs $(\hat{f}_t, \hat{g}_t)$ output by the procedure
	\end{enumerate}
	
		First item: Define the function $\bar{S} := \max\{F_t(x), f(x) - s_{T+1}\}$, so $S = \bar{S}_r$. Using Item 1 of Lemma \ref{lemma:smoothFt}, we see that $\bar{S}(x) \le f(x)$ for all $x \in B(4R)$, and by definition we have $\bar{S}(x) \ge f(x) - s_{T+1}$, thus $|\bar{S}(x) - f(x)| \le s_{T+1}$ for all $x \in B(4R)$. Then using Item 2 of Lemma \ref{lemma:smoothing} we get $|S(x) - f(x)| \le s_{T+1} + \frac{\alpha r^2}{2}$ for all $x \in B(R)$ (we can indeed use this lemma since the definition $r = \sqrt{\eta/\alpha}$ and the assumption $\eta \le \frac{\alpha R^2}{5(T+1)}$ imply that $s_{T+1} \le \alpha R^2$ and $r \le R$). 
		
	Second item: This follows Item 1 of Lemma \ref{lemma:smoothing} instead.
	
	Third item: The subgradients of $F_T$ are (a convex combination of a subset of the) vectors $\tilde{g}_t$, and so $F_T$ is $(\max_t \|\tilde{g}_t\|)$-Lipschitz. Since the vectors came from an $\eta$-approximate oracle for $f$, we have $\|\tilde{g}_t - \nabla f(x_t)\| \le \frac{\eta}{2R}$, and since $f$ is $M$-Lipschitz we get $\|\tilde{g}_t\| \le M + \frac{\eta}{2R}$; it follows that $F_T$ is $(M + \frac{\eta}{2R})$-Lipschitz. Next, the subgradients of $\bar{S}$ come either from subgradients of $F_T$ or gradients of $f$ (or a convex combination thereof), and so $\bar{S}$ is $\max\{M + \frac{\eta}{2R}, M\} = M + \frac{\eta}{2R}$ Lipschitz. Finally, for every $x \in B(R)$ we have ($U$ being uniformly distributed in the unit ball again) 
 $$\|\nabla S(x)\| = \|\E \partial \bar{S}(x + r U)\| \le \E \|\partial \bar{S}(x + r U)\| \le M + \frac{\eta}{2R},$$ 
 where $\partial \bar{S}(x + r U)$ denotes any subgradient at $x + r U$ and the first inequality follows from Jensen's inequality. This proves that $S$ is $(M + \frac{\eta}{2R})$-Lipschitz. 
	
	Fourth item: We need to show that for all $t$, $\hat{f}_t = S(x_t)$ and $\hat{g}_t = \nabla S(x_t)$. By definition, $\hat{f}_t = (F_t)_r(x_t)$ and $\hat{g}_t = \nabla (F_t)_r(x_t)$. Moreover, by Item 3 of Lemma \ref{lemma:smoothFt}, using $F_T$ instead of $F_t$ gives the same quantities, namely $\hat{f}_t = (F_T)_r(x_t)$ and $\hat{g}_t = \nabla (F_T)_r(x_t)$. We claim that for every $t$, $F_T$ and $\bar{S}$ are equal inside the ball $B(x_t, \sqrt{2}\, r)$, which then implies that $\hat{f}_t = (F_T)_r(x_t) = \bar{S}_r(x_t) = S(x_t)$ and $\hat{g}_t = \nabla (F_T)_r(x_t) = \nabla \bar{S}_r(x_t) = \nabla S(x_t)$, as desired. To show the equality in the ball $B(x_t, \sqrt{2}\, r)$, it suffices that the other term in the max defining $\bar{S}$ does not ``cut off'' $F_T$, namely that $f(x) - s_{T+1} \le F_T(x)$ for every $x \in B(x_t, \sqrt{2}\, r)$. But this follows from Item 2 of Lemma \ref{lemma:smoothFt}. 
 
    Substituting the value $r = \sqrt{\eta/\alpha}$ and $s_t = 4 \eta t + \alpha r^2 t$ in the items above concludes the proof of Theorem~\ref{thm:smoothExt}.
	\end{proof}

\section{Separation oracles: proofs of Theorems~\ref{thm:approxSep} and~\ref{thm:approxSep-general}}\label{sec: feasibility}





We now consider the original constrained problem $\min\{f(x) : x \in C \cap X\}$, and show how to run any first-order algorithm $\cA$ using only approximate first-order information about $f$ and approximate separation information from $C$, proving Theorems \ref{thm:approxSep} and \ref{thm:approxSep-general}. The main additional element is to convert the approximate separation information for $C$ into an exact information for a related set $K \approx C$ so it can be used in a black-box fashion by $\cA$, as the previous section did for the first-order information of $f$. For simplicity, we assume throughout that the algorithm $\cA$ only queries points in $B(R)$ (the ball containing the feasible set $C$), since points outside it can be separated exactly.

Given a set of points $x_1, ..., x_T\in B(R)$, we say a sequence of separation responses $r_1, ...,r_T \in \{\feas,\infeas\}\times \R^d$ has a {\em convex extension} if there is a convex set $K \not = \emptyset$ such that there exists an exact (i.e., $0$-approximate) separation oracle for $K$ giving responses $r_1,...,r_T$ for the query points $x_1, ..., x_T$. We will also refer to such responses as {\em consistent with $K$}. As in the previous section, responses from an $\eta$-approximate separation oracle may not by themselves admit a convex extension, and need to be modified in order to allow a consistent, convex extension; for example, approximate separating hyperplanes may not be consistent with a convex set, or may  "cut off" points previously reported as feasible. 
When we say a point $y$ is cut off by a separating hyperplane through $x$ with normal vector $g$, we mean that $\ip{g}{y}> \ip{g}{x}$, i.e. that $y$ is not in the induced halfspace. Note that when given an exact separating hyperplane for some $x\notin C$, no point in $C$ is cut off by it, whereas approximate separating hyperplanes have no such guarantee. With this in mind, we now give a theorem serving as a feasibility analogue to Theorem \ref{thm:LipExt}.

\begin{definition}\label{def:delta-deep-feasible}
    For any convex set $C\subseteq\R^d$ and any $\delta >0$, we define  $C_{-\delta}:=\{x \in C: B(x, \delta) \subseteq C\}$, which will be called {\em $\delta$-deep points of $C$.}
\end{definition}

\begin{theorem}[Online Convex Extensibility] \label{thm:cvxExt}
		Consider a convex set $C\subseteq B(R)$ and a sequence of points $x_1,\ldots,x_T \in B(R)$. There is an online procedure that, given access to an $\eta$-approximate separation oracle for $C$, produces separation responses $\hat{r}_1, ..., \hat{r}_T$ that have a convex extension $K$ satisfying $C_{-\eta} \subseteq K \subseteq C$. Moreover, the procedure only probes the approximate oracle at the points $x_1,\ldots,x_T$.
	\end{theorem}

Note that the guarantee $C_{-\eta} \subseteq K \subseteq C$ means that for any point $x_t$ that is $\eta_C$-deep in $C$, i.e., in $C_{-\eta}$, the response $\hat{r}_t$ produced says $\feas$, whereas for any $x_t\notin C$ it says $\infeas$ and gives a hyperplane separating $x_t$ from $K$ (which cannot cut too deep into $C$, i.e., it contains $C_{-\eta}$). At a high-level, such responses allow one to cut off infeasible solutions, but guarantee that there are still ($\eta_C$-deep) solutions with small $f$-value available.

With this additional procedure at hand, we extend Procedure \ref{proc:optimization} from the main text in the following way to solve constrained optimization: in each step of the procedure, we also send the point $x_t$ queried by the algorithm $\cA$ to the $\eta_C$-approximate separation oracle for $C$, receive the response $\tilde{r}_t$, pass it through Theorem \ref{thm:cvxExt} to obtain the new response $\hat{r}_t$, and send the latter back to $\cA$. We call this procedure \textbf{Approximate-to-Exact-Constr}, and formally state it as follows: 
\begin{mdframed}
    \begin{proc} \textbf{Approximate-to-Exact-Constr}$(\cA, T)$ \label{proc:constrained optimization}
    
    For each timestep $t = 1,\ldots, T$:
    \begin{enumerate}
        \item Receive query point $x_t \in B(R)$ from $\cA$ 
        
        \item Send point $x_t$ to the $\eta_f$-approximate first-order oracle and to the $\eta_C$-approximate separation oracle, and receive the approximate first-order information $(\tilde{f}_t, \tilde{g}_t) \in \R \times \R^d$, and separation response $\tilde{r}_t\in \{\feas\}\cup \{\infeas\}\times \R^d$.
        
        \item Use the online procedures from Theorems \ref{thm:LipExt} and \ref{thm:cvxExt} to construct the first-order pair $(\hat{f}_t, \hat{g}_t)$ and separation response $\hat{r}_t$.
        
        \item Send $(\hat{f}_t, \hat{g}_t), \hat{r}_t$ to the algorithm $\cA$.
    \end{enumerate}
    
    Return the point returned by $\cA$.
\end{proc}
\end{mdframed}

The proof that this procedure yields Theorem \ref{thm:approxSep} is analogous to the one for the unconstrained case of Theorem~\ref{thm:main}, so we only sketch it to avoid repetition.

\begin{proof}[Proof sketch of Theorem \ref{thm:approxSep}]
    Let $F$ and $K$ be the Lipschitz and convex extensions to the answers sent to $\cA$ that are guaranteed by Theorems \ref{thm:LipExt} and \ref{thm:cvxExt}, respectively. \textbf{Approximate-to-Exact-Constr} has the same effect as $\cA$ running on the instance $(F,K)$. One can show that this instance belongs to $\cI(M +\frac{\eta_f}{2R}, R, \rho - \eta_C)$. Then if $err(\cdot)$ is the error guarantee of $\cA$ as in the statement of the theorem, this ensures that we return a solution $\bar{x} \in K \cap X$ satisfying $$F(\bar{x}) \le \OPT(F,K) + err(T, M +\frac{\eta_f}{2R}, R, \rho - \eta_C),$$ where $\OPT(F,K) := \argmin\{F(x) : x \in K\cap X\}$. Since $C_{-\eta_C}$ contains a solution $x$ with value $f(x) \le \OPT(f,C) + \frac{2 MR \eta_C}{\rho}$ (e.g., Lemma 4.7 of \cite{basu2021complexity}), we have $\OPT(f, K) \le \OPT(f,C) + \frac{2 MR \eta_C}{\rho}$. Finally, using the guarantee $\|F - f\|_{\infty} \le 2\eta_f T$, we obtain that $f(\bar{x}) \le \OPT(f,C) + err(T, M +\frac{\eta_f}{2R}, R, \rho - \eta_C) + 4 \eta_f T + \frac{2 MR \eta_C}{\rho}$, concluding the proof of the theorem.
\end{proof}

The proof of Theorem~\ref{thm:approxSep-general} follows effectively the same reasoning as for Theorem~\ref{thm:approxSep}; we also sketch it here. The main difference is that one needs to ensure the optimal solution $x^*$ of \ref{eq:main-prob-flexible} is contained in contained in the auxiliary feasible region $K$ the algorithm $\calA$ uses; otherwise the additional restrictions imposed by $Z$ may lead to arbitrarily bad solutions, or even $K\cap Z$ being empty (consider for example the case of $Z$ being a singleton on the boundary of $C$ that is then cut off by an approximate separation response). However, since $K$ is guaranteed to contain $C_{-\eta_C}$ and we assume that $\eta_C \leq \rho$, the fact that $x^*$ is $\rho$-deep in $C$ implies that it is also in $K$. 
\begin{proof}[Proof sketch of Theorem \ref{thm:approxSep-general}]
    Again, let $F$ and $K$ be the Lipschitz and convex extensions as in the previous proof, so that the instance $(F, K)$ belongs to $\cI(M +\frac{\eta_f}{2R}, R, \rho - \eta_C)$ and $\calA$ returns a solution $\bar{x} \in K \cap Z$ satisfying $F(\bar{x}) \le \OPT(F,K) + err(T, M +\frac{\eta_f}{2R}, R, \rho - \eta_C),$ where $\OPT(F,K) := \argmin\{F(x) : x \in K\cap Z\}$. Recall that $Z$ is assumed to be given and known by the algorithm. Since the optimal solution for $(f, C)$ is assumed to be in $C_{-\rho}$, and $K$ contains $C_{-\eta_C} \supseteq C_{-\rho}$, $K$ contains the optimal solution to to the true instance, $x^*: f(x*) = \OPT(f,C)$. Finally, using the guarantee $\|F - f\|_{\infty} \le 2\eta_f T$, we obtain that $f(\bar{x^*}) \le \OPT(f,C) + err(T, M +\frac{\eta_f}{2R}, R, \rho - \eta_C) + 4 \eta_f T$, concluding the proof of the theorem.
\end{proof}

\subsection{Computing convex-extensible separation responses}

We now prove Theorem \ref{thm:cvxExt}. The result requires the existence of a convex extension $K$ for the responses $\hat{r}_t$ that we construct, and we need $C_{-\eta} \subseteq K \subseteq C$. We provide a procedure that produces the responses $\hat{r}_1,\ldots,\hat{r}_T$ together with sets $K_1, \ldots, K_T$ so that $K_t \cap C$ is consistent with the responses up to this round, i.e., $\hat{r}_1,\ldots,\hat{r}_t$, and is sandwiched between $C_{-\eta}$ and $C$. The set $K_t$ will consist of all the points that were not excluded by the separating hyperplanes of the responses up to this round. Thus, our main task is to ensure that as $K_t$ evolves, it does not exclude the points $x_\tau$ that the responses up to now have reported as \feas (ensuring consistency with previous responses). We also want to ensure that none of the deep points $C_{-\eta}$ is excluded.



    Before stating the formal procedure, we give some intuition on how this is accomplished. Suppose one has $K_{t-1}$ satisfying the desired properties. One receives a new point $x_{t}$ and separation response $\tilde{r}_{t}$ from the approximate oracle, and we need to construct a response $\hat{r}_{t}$ and an updated set $K_{t}$ to maintain the desired properties. 

    Suppose $\tilde{r}_t$ reports that $x_t$ is $\feas$. Our procedure ignores this information, keeps $K_t = K_{t-1}$ and creates a response $\hat{r}_t$ that is \feas if and only if $x_t \in K_t = K_{t-1}$ (also sending a hyperplane separating $x_t$ from $K_t$ if $x_t \notin K_t$; notice that since this hyperplane does not cut into $K_{t-1}$, we do not need to update this set). Notice that $K_t \cap C$ is indeed consistent with the response $\hat{r}_t$.

    \remove{
    If  $x_{t} \in K_{t-1}$, then no adjustment is needed: set $\hat{r}_{t}= (\feas, \star)$, and keep $K_{t} = K_{t-1}$.
    If on the other hand $x_{t}\notin K_{t-1}$, we simply respond $\infeas$ and report a halfsapce separating $x_t$ from $K_{t-1}$, ignoring entirely the $\tilde{r}_{t}$ provided by the oracle, and set $K_{t}= K_{t-1}$. 
    (Notice that $K_t$ still contains $C_{-\eta}$, so the excluded point $x_t$ was ``almost infeasible'' for $C$ and intuitively its exclusion should be ok.) }

    The interesting case is when $\tilde{r}_t$ reports that $x_t$ is \infeas (so $x_t \notin C$) but $x_{t}\in K_{t-1}$. Thus, $x_t$ cannot belong to $K_t \cap C$ (recall we will construct $K_t \subseteq K_{t-1}$), and so to ensure consistency our response $\hat{r}_t$ needs to report \infeas and a separating hyperplane that excludes $x_t$. The first idea is to simply use separating hyperplane reported by the approximate oracle. But this can exclude points that were deemed \feas by our previous responses $\hat{r}_\tau$ (we call these points $\feasbar_{t-1}$), which would violate consistency. Thus, we first rotate this hyperplane as little as possible such that it contains all points in $\feasbar_{t-1}$, and report this rotated hyperplane $H$ (adding it to $K_{t-1}$ to obtain $K_t$). 
    While this rotation protects the points $\feasbar_{t-1}$, we also need to argue that it does not stray too much away from the original approximate separating hyperplane so as to not cut into $C_{-\eta}$.

    We now describe the formal procedure in detail. We use $H(g, \bar{x}) := \{x : \ip{g}{x} \le \ip{g}{\bar{x}}\}$ to denote the halfspace with normal $g$ passing through the point $\bar{x}$. 

	\begin{mdframed}
    \begin{proc} \label{proc:convExt}$\phantom{x}$

    \vspace{-2pt}
    Initialize $K_0 = \R^d$ and $\feasbar_0 = \emptyset$.
    For each $t=1,\ldots,T$:

    \vspace{-6pt}
    \begin{enumerate}[leftmargin=18pt]
        \item  Query the $\eta$-approximate feasibility oracle for $C$ at $x_t$, receiving the response $({flag}_t, \tilde{g}_t)$
         
        \item \vspace{-4pt} \label{item:constrFeas} If ${flag}_t = \feas$ and $x_t \in K_{t-1}$. Define the response $\hat{r}_t = (\feas, \star)$, and set $K_t = K_{t-1}$ and $\feasbar_t = \feasbar_{t-1}\cup {x_t}$.
        
        \item \vspace{-4pt} \label{item:constrEasy} ElseIf  ${flag}_t = \feas$ but $x_t \notin K_{t-1}$. Set $\hat{g}_t$ be any unit vector such that the halfspace $H(\hat{g}_t, x_t)$ contains $K_{t-1}$. Define the response $\hat{r}_t = (\infeas, \hat{g}_t)$. Set $K_t = K_{t-1}$ and  ${\feasbar_t = \feasbar_{t-1}}$.
        
        \item \vspace{-4pt} \label{item:constrProj} Else (so ${flag}_t = \infeas$). Let $\hat{g}_t$ be a unit vector such that the induced halfspace $H(\hat{g}_t, x_t)$ contains $\feasbar_{t-1}$ that is the closest to $\tilde{g}_t$ with this property, i.e.
        \begin{align*}
        	&\|\hat{g}_t - \tilde{g}_t\|_2 \\=& \min_{g \in B(1)} \bigg\{\|g - \tilde{g}_t\|_2 : H(g, x_t) \supseteq \feasbar_{t-1} \bigg\}.
        \end{align*}

       \vspace{-2pt}
       Define the response $\hat{r}_t = (\infeas, \hat{g}_t)$.
       Set $K_t = K_{t-1} \cap H(\hat{g}_t, x_t)$ and $\feasbar_t = \feasbar_{t-1}$.
    \end{enumerate}
		\end{proc}
	\end{mdframed}
	
	We remark that in Line \ref{item:constrProj}, there indeed exists a halfspace supported at $x_t$ that contains all points in $\feasbar_{t-1}$: in this case $x_t \notin C$ (since ${flag}_t = \infeas$) and by definition $\feasbar_{t-1} \subseteq C$, so any halfspace separating $x_t$ from $C$ will do.

	\begin{lemma} \label{lemma:KConvExt}
	For every $t = 1,\ldots,T$, the set $K_t\cap C$, with $K_t$ computed by Procedure \ref{proc:convExt}, is a convex extension to the responses $\hat{r}_1,\ldots,\hat{r}_t$. Moreover, all these sets satisfy $K_t \supseteq C_{-\eta}$.
	\end{lemma}
	
	\begin{proof}
    It suffices to prove this for each iteration, so suppose $K_{t-1}$ with responses $\hat{r}_1, ..., \hat{r}_{t-1}$ satisfy the lemma. If Lines \ref{item:constrFeas} or \ref{item:constrEasy} of the procedure were executed, it is straightforward to see $K_t$ satisfies the lemma. If Line \ref{item:constrProj} executes, $x_t \notin C$, and so the response $\hat{r}_t$ is consistent with $K_t\cap C$ since $K_t = K_{t-1}\cap H(\hat{g}_t, x_t)$. As $\feasbar_{t-1}\subseteq K_t$ by construction and $K_t \subseteq K_{t-1}$, $K_t\cap C$ is consistent with all responses $\hat{r}_1, ..., \hat{r}_t$ made. 
    It remains to show that $K_t$ contains $C_{-\eta}$, for which showing $H(\hat{g}_t, x_t)$ contains it suffices. Notice that since  there exists an exact halfspace $H(g, x_t)$ separating $x_t \notin C$ that is $\frac{\eta}{4R}$-close to $\tilde{g}_t$ (due to the $\eta$-approximate oracle), and $H(g, x_t)$ contains $C$ and thus $\feasbar_{t-1}$, we have $\|\hat{g}_t - \tilde{g}_t\|_2 \leq \frac{\eta}{4R}$. The triangle inequality then reveals $\|\hat{g}_t - g\|_2 \leq \|\hat{g}_t - \tilde{g}_t\|_2+ \|\tilde{g}_t - g\|_2 \leq \frac{\eta}{2R}$, and then it is easy to see that $H(\hat{g}_t, x_t)$ contains $C_{-\eta}$, concluding the proof.
    \end{proof}

\section{A closer comparison with related work} \label{sec: comparison with devolder et al}
In this section, we give a more detailed comparison between our work and the results and settings in closely related work of \cite{Devolder_Nesterov2014_First_order_inexact_oracles}. Therein, the authors define a similar setting using their own notion of inexact first-order oracles and analyze the behavior of a few primal, dual and accelerated gradient methods for convex optimization problems under these oracles. In particular, they show that accelerated gradient methods must accumulate errors in the inexact information setting. As noted, their analysis gives guarantees for specific algorithms, as opposed to our ``universal'' algorithm-independent guarantees. In summary, for the algorithms they provide results for, implementing our method mostly requires less noise to achieve equivalent convergence rates. We begin by comparing the oracles used. 

\paragraph{Comparing the inexact oracles:} The notion of inexact first-order oracle in \cite{Devolder_Nesterov2014_First_order_inexact_oracles} uses two parameters, as opposed to the single $\eta$ parameter in our Definition \ref{def:approxFOO}. 
Nevertheless, an oracle from the Devolder et al. setting corresponds to an oracle in our setting for any family of convex functions on $B(R)$, with appropriate settings of the noise parameters, and vice versa.
We provide the definition for an inexact oracle used by Devolder et al. here: 

 \begin{definition}\label{def: devolder inexact oracle}
 Let $f$ be a convex function on $B(R)$. A first-order $(\delta, L)$-oracle queried at some point $y \in B(R)$ returns a pair $\left(\tilde{f}_y, \tilde{g}_y\right) \in \R \times \R^d$ such that for all $x \in B(R)$ we have
$$
\begin{aligned}
& \tilde{f}_y+\left\langle \tilde{g}_y, x-y\right\rangle \leq f(x) \\
& \leq \tilde{f}_y+\left\langle \tilde{g}_y, x-y\right\rangle+\frac{L}{2}\|x-y\|^2+\delta .
\end{aligned}
$$
\end{definition}


Note that while this definition is valid for any convex function, not just $L$-smooth functions, the motivation for the definition comes from the $L$-smooth inequalities. The parameter $\delta$ can be viewed as the ``noise" while $L$ is the smoothness constant of the family of functions considered, which is taken for granted throughout~\cite{Devolder_Nesterov2014_First_order_inexact_oracles}. We will refer to the oracle of Definition \ref{def:approxFOO} as the $\eta$-approximate oracle and that of Definition \ref{def: devolder inexact oracle} as the $(\delta, L)$-oracle in this section.

Here is what one can say when comparing the two oracles. For the family of $L$-smooth functions, an $\eta$-approximate oracle corresponds to a $(\delta, L)$-oracle with $\delta = 4\eta$ (see Example $b.$ from Section 2.3 in \cite{Devolder_Nesterov2014_First_order_inexact_oracles}). 
For $M$-Lipschitz functions, an $\eta$-approximate oracle is equivalent to a ($\delta$, $L$)-oracle obtained by setting $\delta = 3\eta + \frac{M^2}{L}$ and arbitrary $L > 0$ (if the response of the $\eta$-approximate oracle is $\hat f, \hat g$, the response of the $(\delta,L)$-oracle is $\tilde f = \hat f - \frac{3}{2}\eta$ and $\tilde g = \hat g$).
This follows from the Lipschitz bound $f(y) \leq f(x) + \langle \nabla f(x), y - x \rangle + 2 M \| y -x \|$ combined with the fact that $M r \leq \frac{L}{2} r^2 + \frac{M^2}{2 L}$ for any $r, L > 0$.

In the other direction, if one is given a $(\delta, L)$-oracle from the Devolder et al. setting for any family of convex functions on $B(R)$, this provides an $\eta$-approximate oracle in our sense for $\eta = \max (\delta, 2R\sqrt{2\delta L})$ (see eqn. (8) in \cite{Devolder_Nesterov2014_First_order_inexact_oracles}).

\paragraph{Comparison of the noise levels needed for convergence rates:} Next, we compare the noise levels needed to achieve comparable convergence rates across our approach and the results of Devolder et al. We look at three different function classes for this comparison. 

\begin{itemize}
    \item[a)] Nonsmooth, $M$-Lipschitz functions: As mentioned above, Devolder et al. do not explicitly do an analysis for this family; they focus on $L$-smooth functions. However, one can carry out an analysis of the subgradient method given an $\eta$-approximate oracle similar to the Devolder et al. analysis in the smooth case. The additional error incurred by the inexactness of the oracle is indeed still $O(\eta)$. Therefore, if one uses subgradient descent without any modifications, one needs to set $\eta = O(\frac{1}{T^{0.5}})$ to get overall error $O(\frac{1}{T^{0.5}})$ after $T$ iterations. Using our black-box reduction however, one needs to set $\eta = O(\frac{1}{T^{1.5}})$ for the same convergence rate.

\item[b)] Comparison for $L$-smooth functions with gradient descent: Devolder et al.'s analysis of gradient descent with their notion of $(\delta, L)$ oracle gives $O(\delta)$ additional error. Thus, they need to set $\delta = O(\frac{1}{T})$ to get $O(\frac{1}{T})$ convergence. From the discussion above, an $\eta$-approximate oracle corresponds to a $(\delta, L)$-oracle with $\delta = 4\eta$. In other words, Devolder et al.'s result says that $\eta = O(\frac{1}{T})$ suffices to get $O(\frac{1}{T})$ convergence, if gradient descent is run without modification with an $\eta$-approximate oracle. From our black-box analysis for the $L$-smooth case, we cannot guarantee a convergence rate better than $O(\frac{1}{\sqrt{T}})$, no matter what choice of $\eta$, since we lose the $\sqrt{T}$ factor due to our smoothing technique. Thus, our technique does not achieve the $O(\frac{1}{T})$ rate for gradient descent for $L$-smooth functions.

\item[c)] Comparison for $L$-smooth functions with acceleration: Devolder et al.'s analysis of Nesterov's acceleration with the $(\delta, L)$-oracle gives $O(\delta T)$ additional error. Thus, they need to set $\delta = O(\frac{1}{T^3})$ to get the standard $O(\frac{1}{T^2})$ convergence. From the discussion above, an $\eta$-approximate oracle in our setting corresponds to a $(\delta, L)$-oracle with $\delta = 4\eta$. In other words, Devolder et al.'s result says that $\eta = O(\frac{1}{T^3})$ suffices to get $O(\frac{1}{T^2})$ convergence. To get just $O(\frac{1}{T^{1.5}})$ rate, one can set $\eta = O(\frac{1}{T^{2.5}})$ in the analysis of Devolder et al. (the additional error term from the oracle noise dominates in this case). From our black-box analysis for $L$-smooth functions, we need to set $\eta = O(\frac{1}{T^{2.5}})$ as well, but note that we lose a factor of $\sqrt{d}$ ($d$ being the dimension) additionally. So our final rate is $O(\frac{\sqrt{d}}{T^{1.5}})$, whereas the Devolder et al. analysis does not accrue any dimension-dependent factors.
\end{itemize}

Overall, the algorithm-specific analyses in \cite{Devolder_Nesterov2014_First_order_inexact_oracles} give better convergence rates for equivalent noise levels, or equivalently have less stringent noise requirements to achieve a target convergence rate. This is not too surprising, since our black-box approach is much more general to work with \textit{any} first-order algorithm, and can be viewed as a kind of trade-off to the generality of our results. 


\end{document}